\documentclass{amsart}
\usepackage[utf8]{inputenc}
\usepackage{amsmath, amsthm, latexsym, amssymb, amscd,amsfonts}
\usepackage{tikz-cd}
\usepackage{geometry}
\usepackage{yfonts}
\usepackage{mathtools}  
\usepackage{diffcoeff}  
\usepackage{hyperref}
\usepackage{amsrefs}

\renewcommand{\epsilon}{\varepsilon}
\renewcommand{\phi}{\varphi}

\DeclareMathOperator{\Gal}{Gal}

\DeclareMathOperator{\Div}{Div}
\renewcommand{\div}{\mathrm{div}}

\newcommand{\floor}[1]{\left\lfloor #1 \right\rfloor}

\newcommand{\msim}{{\sim_{\emme}}}

\newcommand{\Gm}{\mathbb{G}_\mathrm{m}}
\newcommand{\Ga}{\mathbb{G}_\mathrm{a}}

\newcommand{\dplus}{\Delta_+}
\newcommand{\dminus}{\Delta_-}

 \newcommand{\emme}{{\mathfrak{m}}}

\renewcommand{\P}{\mathbb{P}}
\newcommand{\Q}{\mathbb{Q}}

\newcommand{\C}{\mathbb{C}}

\newcommand{\Cv}{\mathcal{C}}

\newcommand{\sub}{\subseteq}
\newcommand{\dif}{\setminus}

\newtheorem{thm}{Theorem}[section]
\newtheorem*{teo}{Theorem}

\newtheorem{prop}[thm]{Proposition}
\newtheorem{lem}[thm]{Lemma}
\newtheorem{cor}[thm]{Corollary}

\newtheorem{example}[thm]{Example}
\newtheorem{remark}[thm]{Remark}

\title{Hyperelliptic continued fractions in the singular case of genus zero}
\author{Francesco Ballini}
\address[F. Ballini]{Mathematical Institute, University of Oxford, Woodstock Road OX2 6GG, Oxford, United Kingdom}
\email{ballini@maths.ox.ac.uk}
\author{Francesco Veneziano}
\address[F. Veneziano]{Department of mathematics, University of Genova, Via Dodecaneso~35, 16146 Ge\-no\-va, Italy}
\email{veneziano@dima.unige.it}
\date{\today}

\begin{document}
\begin{abstract}
    It is possible to define a continued fraction expansion of elements in a function field of a  curve by expanding as a Laurent series in a local parameter. Considering the square root of a polynomial $\sqrt{D(t)}$ leads to an interesting theory related to polynomial Pell equations. Unlike the classical Pell equation, the corresponding polynomial equation is not always solvable and its solvability is related to arithmetic conditions on the Jacobian (or generalized Jacobian) of the curve  defined by $y^2=D(t)$. In this setting, it has been shown by Zannier in \cite{zannier} that the sequence of the degrees of the partial quotients of the continued fraction expansion of $\sqrt{D(t)}$ is always periodic, even when the expansion itself is not.
    
    In this article we work out in detail the case in which the curve $y^2=D(t)$ has genus 0, establishing explicit geometric conditions corresponding to the appearance of partial quotients of certain degrees in the continued fraction expansion.
    
    We also show that there are non-trivial polynomials $D(t)$ with non-periodic expansions such that infinitely many partial quotients have degree greater than one.
\end{abstract}
\maketitle

\section{Introduction}

\subsection{Continued fraction expansion in the ring of Laurent series}
Let us fix $\C$ as ground field and let us consider the field $\C((t^{-1}))$ of formal Laurent series in $t^{-1}$, whose elements are expressions of the form
\[
f(t)=\sum_{n=-\infty}^{n_0} c_n t^{n}.
\]

We define the polynomial part of the series $f(t)$ to be
\[
\floor{f}=\sum_{n=0}^{n_0} c_{n}t^n\in\C[t].
\]
Using the same iteration which defines the simple continued fraction expansion of real numbers we can define a continued fraction expansion for elements of $\C((t^{-1}))$ as follows:
define $\alpha_0=f$ and for all $n\geq 0$ set 
\begin{align*}
    a_n&=\floor{\alpha_n},\\
    \alpha_{n+1}&=(\alpha_n-a_n)^{-1}.
\end{align*}
Just as in the classical case, we also define recursively
    \begin{equation}
    \begin{cases}
    p_0=1, \quad p_1=a_0, \\
    p_{n+1}=a_np_n+p_{n-1} \  &\forall\, n\ge 1; \\
    q_0=0, \quad q_1=1, \\ 
    q_{n+1}=a_nq_n+q_{n-1} \ &\forall\, n\ge 1.
    \end{cases}
    \end{equation}
With these definitions $a_n,p_n,q_n$ are polynomials in $t$, and the convergents $p_n(t)/q_n(t)$ satisfy properties of best approximation with respect to the valuation of $\C((t^{-1}))$ similar to the ones seen in the classical case. We quote here a result relevant to our computations.

\begin{prop}[\cite{PT00}*{Proposition 2.1}]\label{Pro:ConvergentCondition}
Let $f\in \C(\!(t^{-1})\!)$ be a Laurent series in the variable $t^{-1}$, and let $p(t)$ and $q(t)$ be coprime polynomials. Then  $p(t)-q(t) f(t) =  O(t^{-\deg q-1})$ holds if and only if $p/q$ is a convergent of the continued fraction expansion of $f$.
\end{prop}

It is an interesting topic to investigate the properties of this continued fraction expansion and compare them with the classical theory of simple continued fractions of real numbers.

A classical theorem of Lagrange states that the real numbers with an eventually periodic expansion are precisely the quadratic irrational numbers. Furthermore, the periodic expansion of $\sqrt{m}$ is strongly related to the fundamental solution of the Pell equation
\[
x^2-my^2=1.
\]

In the function field case the situation is more complicated, but some of these features remain true.

Let $D(t)$ be a polynomial in $\C[t]$ and let us consider the equation
\begin{equation}\label{PolyPell}
x(t)^2-D(t)y(t)^2=1
\end{equation}
to be solved in non-zero polynomials $x(t),y(t)\in\C[t]$. We assume that $D$ is not a square in $\C[t]$, and that it has even degree $2d>0$.

Unlike the classical case, it is not true in general that the polynomial Pell equation \eqref{PolyPell} is always solvable. When this is the case, we say that the polynomial $D(t)$ is \emph{Pellian}.

For example, Euler (\cite{Eu1767}) noticed the identity
\begin{equation}\label{Eqn:EulersPolynomialPellIdentity}
(2t^2+1)^2 - (t^2+1) (2t)^2 = 1
\end{equation}
and more in general  if $T_k$ and $U_k$ denote the Chebyshev polynomials of the first and second kind, the relation
$$T_k(t)^2 -(t^2-1)U_{k-1}(t)^2=1$$
holds.

On the other hand the polynomial $(t^2-1)(t-2)^2$ is not Pellian, as we will see later (see also \cite{ChapterAlpbach}*{Appendix} for some conditions implying non-pellianity).

\medskip

A link between Pell equations and the continued fraction expansion of square roots exists also in the function field case; if we consider $\sqrt{D(t)}$ as a series in $\C((t^{-1}))$, the following theorem of Abel characterizes the pellianity of $D(t)$ in terms of its continued fraction expansion.

\begin{thm}[Abel, 1826]
Let $D(t)\in\C[t]$ be a poynomial of even degree and not a square. Then $D(t)$ is Pellian if and only if the continued fraction expansion of $\sqrt{D(t)}$ is eventually periodic.
\end{thm}

It is a remarkable fact, also considering how recently it has been discovered, that some periodic structure in the sequence of the partial quotients survives even for non-Pellian polynomials. In \cite{zannier} Zannier proved the following theorem:

\begin{thm}[\cite{zannier}*{Theorem 1.1}]
Let $D(t)\in\C[t]$ be a polynomial of even degree and not a square.  Then the sequence of the degrees of the partial quotients in the continued fraction expansion of $\sqrt{D(t)}$ is eventually periodic.
\end{thm}

The methods used to prove this theorem are geometric in nature. 
Just as the theory of the classical Pell equations leads to the study of the field $\Q(\sqrt{d})$, the study of the polynomial Pell equations leads naturally to the field $\C(\Cv)$, where $\Cv$ is the projective curve defined by the equation $y^2=D(t)$.

If we assume that the polynomial $D$ is squarefree, the pellianity of $D$ can be related to the arithmetic of the Jacobian $J$ of the curve $\Cv$. Denoting by $\infty_1$ and $\infty_2$ the two points at infinity of the curve $\Cv$, we have the following classical result.
\begin{thm}[Chebyshev]
Let $D(t)\in\C[t]$ be a squarefree polynomial of even degree and not a square. Then $D$ is Pellian if and only if $(\infty_1)-(\infty_2)$ is torsion in $J$. 
\end{thm}
Similarily, the degrees of the partial quotients in the continued fraction expansion of $\sqrt{D(t)}$ can be studied by looking at the multiples of $(\infty_1)-(\infty_2)$ in $J$.

When $D(t)$ is not squarefree there are technical complications, but it is possible to recover most of the arguments and conclusions seen above, by using the notion of \emph{generalized Jacobians}, which we will introduce in the following section.

\bigskip

In this paper we will study the continued fraction expansion of $\sqrt{D(t)}$ in the case of genus 0, that is to say, when the curve defined by $y^2=D(t)$ is a singular curve of arithmetic genus 0.
Up to linear transformations, this happens if and only if $D(t)=(t^2-1)\cdot D_1(t)^2$.
This case was sketched roughly in \cite{zannier}*{ Example 4.2}; we will work out the computations sketched therein and relate in Theorem~\ref{thm:rootsofunity} the roots of the polynomial $D$ with the degrees of the partial quotients.

We consider first the case in which $D_1$ is a squarefree polynomial, and we consider later in Sections \ref{sec:nonsquarefreeGenJac}--\ref{sec:nonsquarefreeGeo} the case of a $D_1$ with multiple roots.

Our main theorem, proved in Section~\ref{sec:GeoParQuot} is the following
\begin{teo}[Theorem~\ref{thm:rootsofunity}]
Let $\rho_1,\dotsc,\rho_g$ be the roots of $D_1(t)$, and $b_i=\rho_i+\sqrt{\rho_i^2-1}$ for some fixed choices of the square roots. We assume $b_1^{\pm 1}, \dotsc , b_g^{\pm 1}$ to be distinct. Let $r$ be the number of roots of unity among the $b_i$. Then:

\begin{enumerate}
\item Only finitely many partial quotients of $\sqrt{D(t)}$ have degree $\ge r+2$;
\item Let the roots of unity be precisely $b_1, \dotsc , b_r$ and let $D_p(t)=(t^2-1)(t-\rho_1)^2 \dotsc (t-\rho_r)^2$. Then $\sqrt{D(t)}$ has only finitely many partial quotients of degree $\ge d-r$, with the exceptions of the partial quotients $a_n$ corresponding to those indices $n$ such that the convergents $(p_n(t), q_n(t))$ are of the form $(p(t)(t-\rho_{r+1}) \dotsm (t-\rho_g),q(t))$, where $(p(t),q(t))$ are convergents of $\sqrt{D_p(t)}$. In other words, up to finitely many exceptions, every partial quotient with degree $\ge d-r$ is obtained from $\sqrt{D_p(t)}$.
\end{enumerate}
\end{teo}

In order to establish this theorem we determine some special subvarieties in a quotient of the generalized Jacobian of the curve $\Cv$ and compute explicitly their equations; then we study when the powers of a special point belong to these subvarieties.

In Section~\ref{sec:examples} we construct  polynomials $D(t)$ with a non-periodic expansion such that infinitely many partial quotients have degree greater than one, thus correcting an imprecision in \cite{zannier}*{Example 4.1}. Families of examples arise by considering, for instance, a  polynomial $D(t)=H(t^2)$ in which only even exponents of the unknown appear, but there are also many examples not of these special forms. Using the construction seen in Section~\ref{sec:examples} we show that the polynomial
\begin{align*}
 D(t)&=\frac{1}{4096}(t^2-1)(4t^2-1)^2(16t^2+20t+13)^2=\\
&=t^{10}+\frac{5 t^9}{2}+\frac{27 t^8}{16}-\frac{55 t^7}{32}-\frac{911 t^6}{256}-\frac{105 t^5}{64}+\frac{379 t^4}{512}+\frac{505 t^3}{512}+\frac{705 t^2}{4096}-\frac{65 t}{512}-\frac{169}{4096}
\end{align*}
is not Pellian and the sequence of the partial quotients is given by $5,\overline{1,2}$.

\section{Generalized Jacobians}
Generalized Jacobians extend the notion of the Jacobian of a curve and are defined as a quotient of some set of divisors of degree 0 with respect to an equivalence relation finer than linear equivalence. We will need generalized Jacobians in place of ordinary Jacobians because we will study the expansion of non-squarefree polynomials.

The theory of generalized Jacobians had been developed by Rosenlicht and Serre; we refer the reader to Serre's book \cite{serre} for details.

Let us consider a complete smooth algebraic curve $\mathcal{C}$ and an effective divisor $\textfrak{m}$, which we call the \textit{modulus}
$$\textfrak{m} = \sum_{i=1}^n e_i P_i \text{ with } P_i \in \mathcal{C}, \quad e_i \ge 1.$$
We say that two divisors $E,E'$ are $\emme$-\emph{equivalent}, and we write $E\msim E'$, if there exists $f \in \mathbb{C}(\mathcal{C})^*$ such that $$E = \text{div}(f) + E' \text{ and ord}_{P_i}(1-f) \ge e_i.$$
We denote by Div$_\mathfrak{m}^0(\mathcal{C})$ the group of divisors of degree zero whose support is disjoint from $\mathfrak{m}$. The \textit{generalized Jacobian} of the pair $(\mathcal{C},\mathfrak{m})$ is the quotient
$$J_{\mathfrak{m}} \cong \Div_\emme^0(\mathcal{C})/\msim.$$

As shown in \cite{serre}, this is an algebraic group and we have the following exact sequence
$$0 \rightarrow L_{\mathfrak{m}} \rightarrow J_{\mathfrak{m}} \rightarrow J \rightarrow 0$$
where $L_{\mathfrak{m}}$ is an affine algebraic group and $J$ is the usual Jacobian. If $\mathfrak{m}=0$ we have $J_{\mathfrak{m}}=J$; otherwise, $L_{\mathfrak{m}}$ has dimension deg$(\mathfrak{m})-1$.

The algebraic structure of $L_{\mathfrak{m}}$ can be described explicitely (see \cite{serre}, IV § 3) and in our setting $L_{\mathfrak{m}}$ is isomorphic to $\Gm^{n-1} \times \prod_{i=1}^{n} \Ga^{e_i -1}.$

\section{Notation and setting}
Let  $\Cv:\{U^2=T^2-V^2\} \subseteq \P_2$; this is a smooth curve of genus 0. We fix the affine chart given by $V \neq 0$ and use affine coordinates $t=T/V$ and $u=U/V$. With respect to this chart $\Cv$ has two points at infinity, which we will denote by  $(\infty_+)=(1:1:0)$ and $(\infty_-)=(-1:1:0)$.\\

Let $D_1(t)=(t-\rho_1) \dotsm (t-\rho_g) \in  \C[t]$ be a monic polynomial of degree $g$ with all the $\rho_i$ distinct and such that $\rho_i \neq \pm 1$. Let $D(t)=(t^2-1)D_1(t)^2$.

For each $\rho_i$ we define $\xi_i^\pm=(\rho_i,\pm\sqrt{\rho_i^2-1}) \in \Cv$ for some choice of the square roots.\\

We define $\emme=\sum_{i=1}^g \left((\xi_i^+) + (\xi_i^-) \right)$, which is a divisor on $\Cv$, and we denote by $S$ its support. We take $\emme$ as the modulus and write $J_\emme$ for the generalized Jacobian of $\left(\mathcal{C},\emme\right)$.
Notice that the divisor $\emme$ does not depend on the determinations of the square roots $\sqrt{\rho_i^2-1}$, and is invariant upon exchanging $\xi_i^+$ and $\xi_i^-$.\\

Let $\delta$ be the class of the divisor $(\infty_-)-(\infty_+)$ in $J_\emme$. We have

\begin{prop}[\cite{zannier}*{Proposition 2.4}]\label{prop:pellian-gen}
$D(t)$ is pellian if and only if $\delta$ is a torsion point of $J_\emme$.
\end{prop}

In the genus zero case we have an easy description of $J_{\emme}$:

\begin{prop}\label{prop:generalizedjacobian}
$J_{\emme}$ is isomorphic to $\Gm^{2g-1}$
\end{prop}

\begin{proof}
We consider the following map:
\begin{equation*}
\psi: \Div_{\emme}^0(\Cv) \rightarrow \Gm^{2g-1}
\end{equation*}
\begin{equation*}
\psi( \div f)=\left( \frac{f(\xi_1^+)}{f(\xi_1^-)}, \frac{f(\xi_2^-)}{f(\xi_1^-)}, \frac{f(\xi_2^+)}{f(\xi_1^-)}, \dotsc , \frac{f(\xi_g^-)}{f(\xi_1^-)}, \frac{f(\xi_g^+)}{f(\xi_1^-)} \right)
\end{equation*}

This map is well-defined since $\Div_{\emme}^0(\Cv)\subseteq \Div^0(\Cv)$ and $\Cv$ has genus 0, therefore every divisor in $\Div_{\emme}^0(\Cv)$ is of the form $\div f$ for some $f \in \C(\Cv)^*$; this $f$ is unique up to scalar multiplication and the points $\xi_i^{\pm}$ are neither zeros nor poles of $f$ by definition of $\Div_{\emme}^0(\Cv)$. Moreover, $\psi$ is clearly a group homomorphism.

The map $\psi$ is surjective since we can prescribe the values of a function $f$ on each of the $\xi_i^{\pm}$, provided it is not a zero or a pole. 

The kernel of $\psi$ consists in the divisors of the functions $f$ such that $f(\xi_1^+)=f(\xi_1^-) = \dotsb = f(\xi_g^+)=f(\xi_g^-)$. 
As $f$ is defined up to scalar multiplication, we see that the divisors in $\ker\psi$ are precisely those which are equivalent to $0$ under $\sim$, hence $\psi$ factors to an isomorphism $J_\emme \cong \Div_{\emme}^0(\Cv)/{\sim} \cong \Gm^{2g-1}$.
\end{proof}

For our purposes, it will be easier to work with a quotient of $J_\emme$.
It is easy to see that $\Gal(\C(\Cv)/\C(t))$ acts on the points of $\mathcal{C}$ and by linearity on $\Div(\Cv)$ and on $J_\emme$, the nontrivial automorphism $g$ being $(t,u) \mapsto (t,-u)$. Let $\hat{J}_\emme$ be the subgroup of $J_\emme$ invariant for the action above, and define $G=J_\emme/\hat{J}_\emme$. The image of $\delta$ in the quotient $G$ is a torsion point if and only if it is a torsion point in $J_\emme$, since $g(\delta)=-\delta$, so Proposition~\ref{prop:pellian-gen} holds in the same form replacing $J$ by $G$.\\

We will now construct an explicit isomorphism between $G$ and $\Gm^g$.

Let $E\in\Div^0_\emme(\Cv)$ be a divisor of degree zero. Since $\mathcal{C}$ has genus zero, we can write $E=\div(f)$ for some $f\in\C(\Cv)^*$, and we define
\begin{equation*}
i(E)=\left(\frac{f(\xi_{\rho_1}^+)}{f(\xi_{\rho_1}^-)},\dotsc,\frac{f(\xi_{\rho_g}^+)}{f(\xi_{\rho_g}^-)}\right)
\end{equation*}
Arguing as we did for the map $\psi$ above, we see that $i:\Div^0_\emme(\Cv)\to \Gm^g$ is a well-defined group homomorphism.

\begin{prop}\label{prop:smallgeneralizedjacobian}
The map $i:\Div^0_\emme(\Cv)\to \Gm^g$ induces a group isomorphism $\iota:G\to\Gm^g$.
\end{prop}

\begin{proof}
The kernel of the quotient map from $\Div^0_\emme(\Cv)$ to $G$ is generated by the divisors of the functions which are invariant by the action of the Galois group and the functions with value equal to 1 on all the $\xi_i^\pm$. It is clear that these divisors all lie in the kernel of $i$, and therefore the map $\iota$ is well-defined.

To check the injectivity, let $f$ be a rational function such that $\div(f)\in\ker \iota$, so that
$$f(\xi_{\rho_i}^+)=f(\xi_{\rho_i}^-) \quad \forall i=1,\dotsc,g.$$
Therefore we can choose a Galois-invariant function $r(t)\in\C(t)^* \sub \C(\Cv)^*$ that coincides with $f$ on all the $\xi_{\rho_i}^{\pm}$, so that
$$\frac{f}{r}(\xi_{\rho_i}^+)=\frac{f}{r}(\xi_{\rho_i}^-)=1,$$
hence $\div(\frac{f}{r})$ is zero in $J_\emme$. Since $\div(r)$ is zero in $G$ we have that also $\div(f)$ is zero in $G$, which shows the injectivity.

The map $i$ is easily seen to be surjective (e.g. by interpolation as before), and this immediately implies the surjectivity of $\iota$.
\end{proof}

Now we seek the image of $\delta$ in $G$, which is the image of $(\infty_-)-(\infty_+)$. We see that:

\begin{equation*}
\text{div } (t+u) = (\infty_-)-(\infty_+)
\end{equation*}

Hence:

\begin{equation*}
i(\delta)=\left(\dotsc,\frac{(t+u)(\xi_{\rho_i}^+)}{(t+u)(\xi_{\rho_i}^-)},\dotsc\right)=\left(\dotsc,\frac{\rho_i+\sqrt{\rho_i^2-1}}{\rho_i-\sqrt{\rho_i^2-1}},\dotsc\right)=\left(\dotsc,\left(\rho_i+\sqrt{\rho_i^2-1}\right)^2,\dotsc\right)
\end{equation*}

By Chebyshev's theorem (Proposition \ref{prop:pellian-gen}), $D(t)$ is pellian if and only if all the $\rho_i+\sqrt{\rho_i^2-1}$ are roots of unity. These values will be extremely relevant to us, so we define $b_i=\rho_i + \sqrt{\rho_i^2-1}$ and hence $b_i^{-1}=\rho_i - \sqrt{\rho_i^2-1}$ (of course these values depend on the determination of the square roots).

\begin{remark}
$b_i$ is a root of unity if and only if $\rho_i$ is of the form $\cos{\pi r}$ for some $r \in \mathbb{Q}$, so that $D(t)$ is pellian if and only if $D_1(t)$ divides some Chebyshev's polynomial. This can be checked directly by using that the solutions of the Pell equation in this case come from the powers of $t+\sqrt{t^2-1}$.
\end{remark}

\section{On the $W_l$}

\subsection{First maps}

For a point $P \in \Cv \dif S$, we denote with $[P]$ the image of the divisor $(P)-(\infty_+)$ in $G \cong \Gm^g$. We have that:

\begin{equation*}
[P]=(P)-(\infty_+)= \div \left( t+u-(t+u)(P) \right) \text{ as long as } P \neq \infty_+
\end{equation*}
\begin{equation*}
[\infty^+] \text{ is the trivial divisor}
\end{equation*}

Hence the image of $[P]$ via $i$ is:

\begin{equation*}
i([P])=\left(\frac{(t+u)(\xi_1^+)-(t+u)(P)}{(t+u)(\xi_1^-)-(t+u)(P)}, \dotsc , \frac{(t+u)(\xi_g^+)-(t+u)(P)}{(t+u)(\xi_g^-)-(t+u)(P)} \right)
=\left(\frac{(t+u)(P)-b_1}{(t+u)(P)-b_1^{-1}}, \dotsc , \frac{(t+u)(P)-b_g}{(t+u)(P)-b_g^{-1}} \right)
\end{equation*}

The image of $[\infty^+]$ is $(1,\dotsc,1)$. We had already observed that the image of $\delta$ is:

\begin{equation*}
i(\delta)=(b_1^2, \dotsc , b_g^2)
\end{equation*}

\subsection{The $W_l$}

We now define a chain of subvarieties of $G$, namely $0=W_0 \sub \dotsb W_g \sub G$ that will encode the behaviour of the partial quotients of $\sqrt{D(t)}$. We fix a nonnegative integer $l \le g$. We have maps

\begin{equation*}
\begin{tikzcd}[column sep=3.2em]
(\Cv \dif S)^l \arrow[r, "\phi_l"] & G  \arrow[r, "i"] & \mathbb{G}_m^g
\end{tikzcd}
\end{equation*}

so that $\phi_l(P_1,\dotsc,P_l)=[P_1]+\dotsb+[P_l]$ and therefore
$$(i \circ \phi_l) (P_1,\dotsc,P_l)=\left(\dotsc,\prod_{j=1}^l \frac{z_j-b_i}{z_j-b_i^{-1}},\dotsc\right)$$
where $z_j=(t+u)(P_j)$. The map $t+u$ gives an isomorphism between $\Cv$ and $\P_1$: setting $S^*=(t+u)(S)=\{b_1,b_1^{-1}, \dotsc , b_g,b_g^{-1} \}$, we notice that the map

\begin{equation*}
\begin{tikzcd}[column sep=3.2em]
\Cv \setminus S \arrow[r, "t+u"] & \mathbb{P}_1(\mathbb{C}) \setminus S^*
\end{tikzcd}
\end{equation*}

is invertible with inverse function given by

\begin{equation*}
z \rightarrow \left(\frac{1}{2}(z+z^{-1}),\frac{1}{2}(z-z^{-1})\right)
\end{equation*}

Since both $i$ and $t+u$ are isomorphisms, the map $i \circ \phi_l$ is conjugated to the map:

\begin{equation*}
\begin{tikzcd}[column sep=3.2em]
\psi_l: (\mathbb{P}_1(\mathbb{C}) \setminus S^*)^l \arrow[r] & \mathbb{G}_m^g
\end{tikzcd}
\end{equation*}

\begin{equation*}
\psi_l(z_1,\dotsc,z_l)=\left( \dotsc ,\prod_{j=1}^l \frac{z_j-b_i}{z_j-b_i^{-1}}, \dotsc \right)
\end{equation*}

We define $W_l$ to be the image of $\phi_l$. $W_l$ is therefore a constructible set (since the ground field is $\C$).

\subsection{Some properties of the $W_l$}

We notice that $\psi_l$ (respectively $\phi_l$) is invariant by permutation of the $z_j$ (respectively of the $P_j$). We also notice that if, say, $z_2=z_1^{-1}$:
$$\frac{z_1-b_i}{z_1-b_i^{-1}} \frac{z_2-b_i}{z_2-b_i^{-1}}=\frac{b_i^2z_1(z_1-b_i)(b_iz_1-1)}{z_1(1-z_1b_i)(b_i-z_1)}=b_i^2$$
(the equality occuring also if $z_1=0, z_2=\infty$), so that $\psi_l(z,z^{-1},z_3,\dotsc,z_l)$ does not depend on $z$ (respectively, $\phi_l(P,\sigma(P),P_3,\dotsc,P_l)$ does not depend on $P$, $\sigma$ being the nontrivial element of $\Gal \mathbb{C}(\Cv)/\mathbb{C}(t)$: we have that $(t+u)(P) (t+u)(\sigma(P))=(t^2-u^2)(P)=1)$. The next proposition states that these are all the possible symmetries of $\psi_l$.

\begin{prop}
$\psi_l(z_1,\dotsc,z_l)=\psi_l(z_1',\dotsc,z_l')$ if and only if there exists $k \le \frac{l}{2}$ such that, up to a permutation of the $z_j$ and $z_j'$, the following equalities hold:
\begin{align*}
   z_1 z_2 &=z_3z_4=\dotsb=z_{2k-1}z_{2k}=z_1'z_2'=z_3'z_4'=\dotsb=z_{2k-1}'z_{2k}'=1,\\
z_j&=z_j'\text{ for }j > 2k. 
\end{align*}

\end{prop}

\begin{proof}
The \textit{if} part follows from the observations above.\\
For the other direction, if, say, $z_1z_2=z_1'z_2'=1$, we have that
$$(b_1^2,\dotsc,b_g^2) \cdot \psi_{l-2}(z_3,\dotsc,z_l)=\psi_l(z_1,\dotsc,z_l)=\psi_l(z_1',\dotsc,z_l')=(b_1^2,\dotsc,b_g^2) \cdot \psi_{l-2}(z_3',\dotsc,z_l')$$
so we can assume $z_j \neq z_k^{-1}$ for any $1 \le j,k \le l$. Moreover, if for some $1 \le j,k \le l$ we have $z_j'=0$ and $z_k'=\infty$, we change them with $z_j'=z_k'=1$.\\
Let $\alpha$ be the number of $z_j$ which are neither $0$ nor $\infty$ (and we assume they are $z_1,\dotsc,z_{\alpha}$), $\beta$ the number of $0$ among the $z_j$ and $\gamma$ the number of $\infty$ among the $z_j$. We do the same with $\alpha',\beta',\gamma'$ for the $z_j'$. Notice that $\alpha+\beta+\gamma=\alpha'+\beta'+\gamma'=l$. Let us call $P(x)=x^{\beta}\prod_{j=1}^{\alpha} (x-z_j)$ and $P'(x)=x^{\beta'}\prod_{j=1}^{\alpha'} (x-z_j')$. The hypothesis is now equivalent to
$$\frac{P(b_i)}{P(b_i^{-1})}=\frac{P'(b_i)}{P'(b_i^{-1})} \text{ for all } 1 \le i \le g$$
since the value of $\frac{z-b_i}{z-b_i^{-1}}$ at $\infty$ is $1$. By defining $Q(x)=P(x^{-1})x^{\alpha+\beta}$ and $Q'(x)=P'(x^{-1})x^{\alpha'+\beta'}$ (which are the reciprocal polynomials of $P(x)$ and $P'(x)$ and have degrees respectively $\alpha$ and $\alpha'$), we see that the hypothesis is equivalent to
$$F(x)=x^{\gamma'}P(x)Q'(x)-x^{\gamma}P'(x)Q(x) \text{ having as zeroes all the } b_i,b_i^{-1}.$$
We have that $\deg F(x) \le \max{(\alpha+\beta+\alpha'+\gamma',\alpha'+\beta'+\alpha+\gamma)} \le 2l \le 2g$, but $F(x)$ has $2g+2$ zeros (the $b_i^{\pm 1}$ plus $\pm 1$), so that $F(x)=0$. We recover
$$x^{\gamma'}P(x)Q'(x)=x^{\gamma}P'(x)Q(x)$$
$$x^{\gamma'+\beta}\prod_{j=1}^{\alpha}(x-z_j)\prod_{j=1}^{\alpha'}(x-z_j'^{-1})=x^{\gamma+\beta'}\prod_{j=1}^{\alpha'}(x-z_j')\prod_{j=1}^{\alpha}(x-z_j^{-1})$$
so that $\gamma'+\beta=\gamma+\beta'$. Our first assumptions imply that either $\beta$ or $\gamma$ is zero and that either $\beta'$ or $\gamma'$ is zero. Any of the four possibilities implies that $\beta=\beta'$ and $\gamma=\gamma'$ (and hence $\alpha=\alpha'$) and therefore
$$\prod_{j=1}^{\alpha}(x-z_j)\prod_{j=1}^{\alpha}(x-z_j'^{-1})=\prod_{j=1}^{\alpha}(x-z_j')\prod_{j=1}^{\alpha}(x-z_j^{-1})$$
Our assumptions also imply that $z_j \neq z_k^{-1}$ for all $1 \le j,k \le \alpha$. By forcing the LHS and RHS to have the same zeroes, we get (up to permutation), $z_j'=z_j$ for $1 \le j \le \alpha$.
\end{proof}

\begin{cor}\label{cor:dimension}
$\dim{i(W_l)}=l$
\end{cor}

\begin{proof}
The set $i(W_l)$ is the image of $\psi_l$, so $\dim{i(W_l)} \le \dim{(\mathbb{P}_1(\mathbb{C}) \setminus S^*)^l}=l$. But $\psi_l$ is injective over a dense open subset of the domain, namely the complement of the union of algebraic sets defined by $z_jz_k-1$ for $1 \le j,k \le l$, so we get equality.
\end{proof}

We are then in the following situation:

\begin{equation*}
0=i(W_0) \sub \dotsb \sub i(W_g) \sub G
\end{equation*}

Moreover $\overline{i(W_g)}=G$ since $G$ is irreducible and they have the same dimension. We have that $W_l \sub W_{l+1}$ since $[\infty_+]$ is the trivial divisor.

\subsection{Equations for $\overline{i(W_{g-1})}$}
We can provide explicit equations for the closure of $\overline{i(W_{g-1})}$.\\

Given complex numbers $u_1, \dotsc ,u_g$, we define the antisymmetric function:

\begin{equation*}
H(u_1, \dotsc ,u_g)=\prod_{i<j} (u_i-u_j)
\end{equation*}

We denote with $A$ a generic element of $\{ 0, 1 \}^g$, with $a_i$ its $i$-th projection and with $a=\sum_{i=1}^g a_i$. We claim that the equation for $\overline{i(W_{g-1})}$ is:

\begin{equation*}
f(x_1, \dotsc ,x_g)=\sum_{A} (-1)^a H(b_1^{1-2a_1}, \dotsc ,b_g^{1-2a_g}) x_1^{a_1} \dotsc x_g^{a_g}
\end{equation*}

\begin{lem}\label{lem:firreducible}
The polynomial $f(x_1, \dotsc ,x_g)$ is irreducible.
\end{lem}

\begin{proof}
$f$ is linear in all its variables, so, if it was reducible, it would be a product of polynomials $f_1$ and $f_2$ both linear in their variables and moreover with distinct appearing variables. For instance, if $x_1$ appears in $f_1$ with coefficient $c_1$ and $x_2$ in $f_2$ with coefficient $c_2$, then the coefficient of $x_1x_2$ in $f$ is $c_1c_2$. Considering also the constant terms of $f_1$ and $f_2$, we conclude that, for $f$, the product of the coefficients of $x_1$ and $x_2$ equals the product of the coefficient of $x_1x_2$ and its constant term. By writing this down:
\begin{equation*}
H(b_1,b_2, \dotsc ,b_g)H(b_1^{-1},b_2^{-1}, \dotsc ,b_g)=H(b_1^{-1},b_2,\dotsc,b_g)H(b_1,b_2^{-1}, \dotsc ,b_g)
\end{equation*}
\begin{equation*}
LHS=(b_1-b_2)(b_1^{-1}-b_2^{-1}) \prod_{i=3}^g \left((b_1-b_i)(b_2-b_i)(b_1^{-1}-b_i)(b_2^{-1}-b_i)\right) \prod_{3=i<j=g}(b_i-b_j)
\end{equation*}
\begin{equation*}
RHS=(b_1^{-1}-b_2)(b_1-b_2^{-1}) \prod_{i=3}^g\left((b_1^{-1}-b_i)(b_2-b_i)(b_1-b_i)(b_2^{-1}-b_i)\right) \prod_{3=i<j=g}(b_i-b_j)
\end{equation*}
\begin{equation*}
(b_1-b_2)(b_1^{-1}-b_2^{-1})=(b_1^{-1}-b_2)(b_1-b_2^{-1})
\end{equation*}
\begin{equation*}
(b_1-b_2)^2=(b_1b_2-1)^2
\end{equation*}
\begin{equation*}
(b_1-1)(b_1+1)(b_2-1)(b_2+1)=0
\end{equation*}
which is impossibile.
\end{proof}

\begin{lem}\label{lem:distinctzeros}
Let $n,m \ge 1$ be integers, and let $p(x_1, \dotsc ,x_n)$ be a polynomial with complex coefficients with degree at most $m$ in each variable. Let $S$ be a set of complex numbers such that $S$ has at least $m+n$ elements. Suppose that for any choice of \textit{distinct} elements $u_1, \dotsc ,u_n \in S$ we have $p(u_1, \dotsc ,u_n)=0$. Then $p(x_1, \dotsc ,x_n)$ is the zero polynomial.
\end{lem}

\begin{proof}
We prove this by induction on $n$. If $n=1$ then $p$ is a polinomial of degree at most $m$ which vanishes on $m+1$ complex numbers, so it is identically zero. If $n>1$, let:

\begin{equation*}
p(x_1, \dotsc ,x_n)=\sum_{i=0}^m x_n^i p_i(x_1, \dotsc ,x_{n-1})
\end{equation*}

where for any $0 \le i \le n$ the polynomial $p_i(x_1, \dotsc ,x_{n-1})$ has degree at most $m$ in each variable. Fix a choice of distinct $u_1, \dotsc ,u_{n-1} \in S$. Then $p(u_1, \dotsc ,u_{n-1},x_n)$, as a polynomial in $x_n$, is vanished by hypothesis by the at least $m+1$ remaining elements of $S$, so that it is the zero polynomial and then $p_i(u_1, \dotsc ,u_{n-1})=0$. By varying $u_1, \dotsc ,u_{n-1}$ and using the inductive hypothesis we see that $p_i(x_1, \dotsc ,x_{n-1})=0$, so we're done.
\end{proof}

\begin{lem}\label{lem:containedness}
$i(W_{g-1})$ is contained in the zero locus of $f(x_1, \dotsc ,x_g)$.
\end{lem}

\begin{proof}
We have $x_i=\prod_{j=1}^{g-1} \frac{z_j-b_i}{z_j-b_i^{-1}}$ for $1 \le i \le g$, so that our equation becomes:

\begin{equation*}
\sum_{A} (-1)^a H(b_1^{1-2a_1}, \dotsc ,b_g^{1-2a_g}) \prod_{i=1}^g \prod_{j=1}^{g-1} \left(\frac{z_j-b_i}{z_j-b_i^{-1}}\right)^{a_i}
\end{equation*}

Clearing the denominators we get:
$$p(z_1, \dotsc ,z_{g-1})=\sum_{A} (-1)^a H(b_1^{1-2a_1}, \dotsc ,b_g^{1-2a_g}) \prod_{i=1}^g \prod_{j=1}^{g-1} (z_j-b_i^{2a_i-1})$$
We see that $p(z_1, \dotsc ,z_{g-1})$ is a polynomial of degree at most $g$ in each variable and we hope to apply Lemma~\ref{lem:distinctzeros} with $S= \{ b_1^{\pm 1}, \dotsc ,b_g^{\pm 1} \}$.

First of all, if both, say, $b_1$ and $b_1^{-1}$ are chosen, then

\begin{equation*}
\prod_{j=1}^{g-1} (z_j-b_1^{2a_1-1})=0,
\end{equation*}
so we can assume not both $b_i$ and $b_i^{-1}$ being chosen. We assume $z_1=b_1, \dotsc ,z_{g-1}=b_{g-1}$, the other cases being analogous by the usual symmetry between $b_i$ and $b_i^{-1}$.
\begin{equation*}
p(b_1, \dotsc,b_{g-1})=\sum_{A} (-1)^a H(b_1^{1-2a_1}, \dotsc ,b_g^{1-2a_g}) \prod_{i=1}^g \prod_{j=1}^{g-1} (b_j-b_i^{2a_i-1})
\end{equation*}
so that a summand on the right is zero whenever $a_i=1$ for some $1 \le i \le g-1$. So the only nonzero possibilities for $A$ are $a_0= \dotsb =a_{g-1}=0,a_g=0,1$:

\begin{equation*}
p(b_1, \dotsc ,b_{g-1})=H(b_1, \dotsc ,b_{g-1},b_g)\prod_{i=1}^{g-1} \prod_{j=1}^{g-1} (b_j-b_i^{-1}) \cdot \prod_{j=1}^{g-1} (b_j-b_g^{-1})-H(b_1, \dotsc ,b_{g-1},b_g^{-1})\prod_{i=1}^{g-1} \prod_{j=1}^{g-1} (b_j-b_i^{-1}) \cdot \prod_{j=1}^{g-1} (b_j-b_g)
\end{equation*}

The right hand side is then a multiple of:

\begin{equation*}
H(b_1, \dotsc ,b_{g-1},b_g) \prod_{j=1}^{g-1} (b_j-b_g^{-1})-H(b_1, \dotsc ,b_{g-1},b_g^{-1}) \prod_{j=1}^{g-1} (b_j-b_g)
\end{equation*}
\begin{equation*}
\prod_{i<j \le g-1} (b_i-b_j) \left( \prod_{j=1}^{g-1}(b_j-b_g)\prod_{j=1}^{g-1} (b_j-b_g^{-1})-\prod_{j=1}^{g-1}(b_j-b_g^{-1})\prod_{j=1}^{g-1} (b_j-b_g)\right)=0\qedhere
\end{equation*}
\end{proof}

By combining this Lemma, the irreducibility of $f$ and the fact that $i(W_{g-1})$ has dimension $g-1$ (Corollary \ref{cor:dimension}), we get that the desired equation is precisely $f(x_1, \dotsc , x_g)$.

\section{On the closure of the $W_l$}

The points in $W_l$ have an important geometrical description: they are the divisors of the form $[P_1]+ \dotsb + [P_l]$ for $P_1, \dotsc , P_l \in \Cv \dif S$. We need a similar description for the closure $\overline{W_l}$ and it turns out that such a description can be given in terms of the sets $W_l$ attached to a different polynomial $D(t)$.\\

We recall that $D(t)=(t^2-1)D_1(t)^2$ with $D_1(t)=(t-\rho_1) \dotsm (t-\rho_g)$. Let us define
\begin{equation*}
D_1^{(j)}(t)=(t-\rho_1) \dotsm (t-\rho_{j-1})(t-\rho_{j+1}) \dotsm (t-\rho_g)
\end{equation*}
the polynomial obtained omitting the $j$-th factor from $D_1(t)$,
so that $D_1(t)=D_1^{(j)}(t) (t-\rho_j)$ and we define $D^{(j)}(t)=(t^2-1)(D_1^{(j)}(t)^2)$. We denote with the superscript $(j)$ all the quantities associated to $D^{(j)}(t)$ corresponding to those associated to $D(t)$. In particular, the $\xi_i^{\pm}$ associated to $D^{(j)}(t)$ will be the same associated to $D(t)$ with the exception of $\xi_j^{\pm}$, which are omitted. We use the same determination of the square roots $\sqrt{\rho_i^2-1}$ for $i \neq j$, associated the common roots for $D(t)$ and $D^{(j)}(t)$.\\

Let $G^{(j)}$ be defined the same way as $G$, but for $D^{(j)}(t)$. Having $\emme^{(j)}$ the divisor such that $\emme = (\xi_j^+) + (\xi_j^-) + \emme^{(j)}$, we have an isomorphism with $\Gm^{g-1}$, given by the map:

\begin{equation*}
i^{(j)}: \Div_{\emme^{(j)}}^0(\Cv) \rightarrow \Gm^{g-1}
\end{equation*}
\begin{equation*}
i^{(j)}(\div f)= \left( \frac{f(\xi_1^+)}{f(\xi_1^-)}, \dotsc, \frac{f(\xi_{j-1}^+)}{f(\xi_{j-1}^-)}, \frac{f(\xi_{j+1}^+)}{f(\xi_{j+1}^-)}, \dotsc, \frac{f(\xi_g^+)}{f(\xi_g^-)} \right)
\end{equation*}

Moreover, taking a divisor $\div f \in \Div_{\emme}^0(\Cv) \sub \Div_{\emme^{(j)}}^0(\Cv)$, its images in $G$ and $G^{(j)}$ are related by the projection $\pi_j: G \rightarrow G^{(j)}$ which forgets the $j$-th coordinate (as a map from $\Gm^{g}$ to $\Gm^{g-1}$).\\

We can now give a description the closure of $W_l$.

\begin{prop}\label{prop:closurewl}
$\overline{W_l}$ is the union of $W_l$ and 
\begin{equation*}
\delta \bigcup_{j=1}^g \pi_j^{-1} \left(\overline{W^{(j)}_{l-2}}\right)
\end{equation*}
\end{prop}

\begin{remark}\label{remark:w-1}
We set $W_{-1}$ (and $W_{-2}$) to be empty. In fact, $W_1$ is closed, since its limit points outside of $W_1$ would necessarily correspond to $z_1=b_i^{\pm}$, that is, either a zero or a pole of $\frac{z_1-b_i}{z_1-b_i^{-1}}$.
\end{remark}

\begin{proof}
It is clear that $W_l \sub \overline{W_l}$. Let us show that, say, $\delta \pi_1^{-1} \left(\overline{W^{(1)}_{l-2}}\right) \sub \overline{W_l}$. We consider everything in $\Gm^{g}$ and $\Gm^{g-1}$ via $i$ and $i^{(1)}$. Since our ground field is $\C$, our sets are constructible and, in particular, the analytic closure is the same as the Zariski closure. The points of $W^{(1)}_{l-2}$ are of the form:

\begin{equation*}
\left( \prod_{k=1}^{l-2} \frac{z_k-b_2}{z_k-b_2^{-1}}, \dotsc , \prod_{k=1}^{l-2} \frac{z_k-b_g}{z_k-b_g^{-1}} \right)
\end{equation*}

For $z_1,\dotsc,z_{l-2}$ in $\P_1(\C) \dif S^{*(1)}=\P_1(\C) \dif \{b_2^{\pm 1}, \dotsc , b_g^{\pm 1} \}$. The points of $\overline{W^{(1)}_{l-2}}$ are hence limits (with respect to the euclidean metric) of:

\begin{equation*}
\left( \prod_{k=1}^{l-2} \frac{z_k^{(n)}-b_2}{z_k^{(n)}-b_2^{-1}}, \dotsc , \prod_{k=1}^{l-2} \frac{z_k^{(n)}-b_g}{z_k^{(n)}-b_g^{-1}} \right)
\end{equation*}

For sequences $z_1^{(n)},\dotsc,z_{l-2}^{(n)}$ in $\P_1(\C) \dif S^{*(l)}$. By passing to subsequences, since $\P_1(\C)$ is compact, we can assume that $z_k^{(n)}$ converges to some value in $\P_1(\C)$ (but it might be in $S^{*(l)}$). Moreover, since none of the $b_i^{\pm 1}$ is the same as $b_1^{\pm 1}$ for $i \ge 2$, we can slightly alter the sequences $z_k^{(n)}$ so that no value is exactly $b_1$ or $b_1^{-1}$. We claim that all the points of $\delta \pi_1^{-1} \left(\overline{W^{(1)}_{l-2}}\right)$ can be expressed as a limit of:

\begin{equation*}
\left( \frac{x^{(n)}-b_1}{x^{(n)}-b_1^{-1}} \cdot \frac{y^{(n)}-b_1}{y^{(n)}-b_1^{-1}} \cdot \prod_{k=1}^{l-2}  \frac{z_k^{(n)}-b_1}{z_k^{(n)}-b_1^{-1}}, \dotsc ,
\frac{x^{(n)}-b_i}{x^{(n)}-b_i^{-1}} \cdot \frac{y^{(n)}-b_i}{y^{(n)}-b_i^{-1}} \cdot
\prod_{k=1}^{l-2} \frac{z_k^{(n)}-b_i}{z_k^{(n)}-b_i^{-1}}, \dotsc \right)
\end{equation*}

For sequences $z_k^{(n)},x^{(n)},y^{(n)}$ in $\P_1(\C) \dif S^*$. This is an element of $\overline{W_l}$. If we have sequences $x^{(n)}$ and $y^{(n)}$ that converge to $b_1$ and $b_1^{-1}$ respectively, we have:

\begin{equation*}
\frac{x^{(n)}-b_i}{x^{(n)}-b_i^{-1}} \cdot \frac{y^{(n)}-b_i}{y^{(n)}-b_i^{-1}} \rightarrow b_i^2
\end{equation*}

Since every point of $\overline{W^{(1)}_{l-2}}$ is obtained as a limit as above, we choose the same sequences $z_k^{(n)}$ (avoiding all the $b_i^{\pm 1}$, $b_1^{\pm 1}$ included) and we obtained the desired point for the last $g-1$ coordinates. We can fix the first coordinate since we can decide how the $x^{(n)}$ and $y^{(n)}$ approach $b_1$ and $b_1^{-1}$: for instance, if $x^{(n)}=b_1+X_n$ and $y^{(n)}=b_1^{-1}+Y_n$, we have that (consider $X_n$ and $Y_n$ small):

\begin{equation*}
\frac{x^{(n)}-b_1}{x^{(n)}-b_1^{-1}} \cdot \frac{y^{(n)}-b_1}{y^{(n)}-b_1^{-1}} \cdot \prod_{k=1}^{l-2}  \frac{z_k^{(n)}-b_1}{z_k^{(n)}-b_1^{-1}}= \frac{X_n}{Y_n} \cdot \frac{-b_1+b_1^{-1}+Y_n}{b_1-b_1^{-1}+X_n} \cdot \prod_{k=1}^{l-2}  \frac{z_k^{(n)}-b_1}{z_k^{(n)}-b_1^{-1}}
\end{equation*}

So that a suitable choice of small $X_n,Y_n$ makes the coordinate converge to any prescribed nonzero complex number (notice that $b_1 \neq b_1^{-1}$, since the $b_i$ are never $\pm 1$). This proves one inclusion.\\

For the other inclusion, consider a point of $\overline{W_l} \dif W_l$. This must be a limit of the form:

\begin{equation*}
\left( \prod_{k=1}^{l} \frac{z_k^{(n)}-b_1}{z_k^{(n)}-b_1^{-1}}, \dotsc , \prod_{k=1}^{l} \frac{z_k^{(n)}-b_i}{z_k^{(n)}-b_i^{-1}}, \dotsc \right)
\end{equation*}

With the $z_k^{(n)}$ contained in $\P_1(\C) \dif S$. As before, we can assume $z_k^{(n)}$ converging to some value in $\P_1(\C)$. If none of these values lies in $S$, then such limit is actually an element of $W_l$ (just choose the $z_k$ to be the limits of the $z_k^{(n)}$). So, say, assume that $z_1^{(n)}$ converges to $b_1$. Since the limit of the first coordinate is a nonzero complex number, necessarily one of the limits of the other sequences is $b_1^{-1}$, hence we can assume that $z_2^{(n)}$ converges to $b_1^{-1}$. As before, we have that, for $i \ge 2$:

\begin{equation*}
\frac{z_1^{(n)}-b_i}{z_1^{(n)}-b_i^{-1}} \cdot \frac{z_2^{(n)}-b_i}{z_2^{(n)}-b_i^{-1}} \rightarrow b_i^2
\end{equation*}

Thus, the projection $\pi_1$ sends our limit point to:

\begin{equation*}
\left( b_2^2 \prod_{k=3}^{l} \frac{z_k^{(n)}-b_2}{z_k^{(n)}-b_2^{-1}}, \dotsc , b_g^2 \prod_{k=3}^{l} \frac{z_k^{(n)}-b_g}{z_k^{(n)}-b_g^{-1}} \right)
\end{equation*}

Which is a point of $\overline{W_{l-2}^{(1)}}$ multiplied by the last $g-1$ components of $\delta$. This proves the reverse inclusion.
\end{proof}

\begin{cor}\label{corollary:powersofdelta}
If $\delta^n \in \overline{W_l} \dif W_l$, then there is some $j$ such that $\delta^{(j)n-1} \in \overline{W^{(j)}_{l-2}}$.
\end{cor}

\begin{proof}
If $\delta^n \in \overline{W_l} \dif W_l$, then by Proposition \ref{prop:closurewl}, there is $j$ such that $\delta^{n-1} \in \pi_{j}^{-1} \left( \overline{W^{(j)}_{l-2}} \right)$. This is equivalent to the thesis.
\end{proof}

\section{Geometry of partial quotients}\label{sec:GeoParQuot}

\subsection{Preliminary lemmas}
Let us first recall the following result, which is explained in  \cite{PT00}*{a few lines before Proposition 2.1}. Let $(p_n(t),q_n(t))$ be the sequence of convergents of a function $f(t) \in \C((t^{-1}))$ and let $a_n(t)$ be the corresponding partial quotients. Then

\begin{lem}\label{lem:qnqn1}
$p_n(t)-q_n(t)f(t) \text{ vanishes at infinity with order } \deg q_n(t) + \deg a_n(t) = \deg q_{n+1}(t)$
\end{lem}

Let us now get back to our $D(t)=(t-1)^2(t-\rho_1)^2 \dots (t-\rho_g)^2$. We set $d=g+1$. Any regular function on the affine part of $\mathcal{C}$ is of the form $p(t)-uq(t)$, where $p(t),q(t) \in \mathbb{C}[t]$. We have an easy criterion to determine whether $D_1(t)$ divides $q(t)$.

\begin{lem}\label{lem:pre1}
Let $f(t,u)=p(t)-uq(t)$ as above and suppose that the support of $\div f$ is disjoint from $\mathfrak{m}$. Then $D_1(t)$ divides $q(t)$ if and only if the image of $\div f$ in $G$ is zero.
\end{lem}

\begin{proof}
Suppose $D_1(t)$ divides $q(t)$. Then by using the isomorphism $i$:

\begin{equation*}
\left(\dotsc,\frac{f(\xi_{\rho_i}^+)}{f(\xi_{\rho_i}^-)},\dotsc\right)=\left(\dotsc,\frac{p(\rho_i)-\sqrt{\rho_i^2-1} \cdot q(\rho_i)}{p(\rho_i)+\sqrt{\rho_i^2-1} \cdot q(\rho_i)},\dotsc\right)=\left(\dotsc,\frac{p(\rho_i)}{p(\rho_i)},\dotsc\right)=(1,\dotsc,1)
\end{equation*}

Vice-versa, if the image is zero, then

\begin{equation*}
p(\rho_i)+\sqrt{\rho_i^2-1} \cdot q(\rho_i) = p(\rho_i)-\sqrt{\rho_i^2-1} \cdot q(\rho_i) \text{ for every } i
\end{equation*}

and, since $\rho_i \neq \pm 1$, we have the thesis.
\end{proof}

We now want to detect the couples $(p,q)$ which are convergents of $\sqrt{D(t)}$. We have the following:

\begin{lem}\label{lem:pre2}
Let $(p_n(t),q_n(t))$ be the convergents of $\sqrt{D(t)}$ and let $a_n(t)$ be its partial quotients. We have that

\begin{equation*}
\div (p_n(t)-uD_1(t)q_n(t))=(\deg q_n(t) + \deg a_n(t))(\infty_+)-(\deg q_n(t) + d)(\infty_-)+\sum_{i=1}^{d- \deg a_n} (x_i)
\end{equation*}

with $x_i$ on the affine part of $\mathcal{C}$
\end{lem}

\begin{proof}
Expressing $p_n(t)-uD_1(t)q_n(t)$ as a power series at $(\infty_+)$ (with parameter $\frac{1}{t}$) we get the expansion of
\begin{equation*}
p_n(t)-\sqrt{D(t)}q_n(t)
\end{equation*}

which vanishes at $(\infty_+)$ with order $\deg q_n(t) + \deg a_n(t)$. Its only pole is then at $(\infty_-)$, say with order $e$, since it is a regular function on the affine part of $\mathcal{C}$. Its conjugate (with respect to $(t,u) \rightarrow (t,-u)$) has the same property with $(\infty_+)$ instead and by summing them we have:
\begin{equation*}
\div (2p_n(t))=-e(\infty_+)-e(\infty_-)+ \text{ zeroes}
\end{equation*}

So $2e=$ deg $p_n(t)$ and $e=\text{deg }q_n(t) + d$. The remaining zeroes are forced to be on the affine part.
\end{proof}

\begin{remark}\label{rmk:pelliandegree}
Note that this proof implies that deg $a_n(t) \le d$. Moreover, if for some $n$ we have equality, then $\ div(p_n(t)-uD_1(t)q_n(t))$ has support disjoint from $\mathfrak{m}$ and has zero image in $G$, so that the $(\text{deg }q_n(t) + d)$-th power of $\delta$ is zero. In other words, equality occurs if and only if $D(t)$ is pellian (the other arrow being obvious).
\end{remark}

\begin{remark}
It is not guaranteed that the $(x_i)$ do not belong to the support of $\mathfrak{m}$, but if it was so, say $x_i=(\rho_i,\sqrt{\rho_i^2-1})$, we would have $p_n(\rho_i)=0$, so that $\left( \frac{p_n(t)}{t-\rho_i},q_n(t) \right)$ would be a convergent of $\frac{\sqrt{D(t)}}{t-\rho_i}$.
\end{remark}

We have a corresponding vice-versa lemma:

\begin{lem}\label{lem:pre3}
Let $p(t),q(t) \in \mathbb{C}[t]$ such that the order of $p(t)-uD_1(t)q(t)$ at $(\infty_+)$ is deg $q(t)+l$ for some positive integer $l$. Then $(p(t),q(t))$ is of the form $(r(t)p_n(t),r(t)q_n(t))$ where $(p_n(t),q_n(t))$ are convergents of $\sqrt{D(t)}$ and deg $a_n(t) \ge l$. In particular, if $p(t)$ and $q(t)$ are coprime, then $p(t)/q(t)$ is a convergent of $\sqrt{D(t)}$.
\end{lem}

\begin{proof}
Expressing the series as before, we have that

\begin{equation*}
p(t)-\sqrt{D(t)}q(t)
\end{equation*}

vanishes at $(\infty_+)$ with order deg $q(t)+l$. Any polynomial can be written uniquely in the form
$$\sum_{i=0}^n r_i(t)q_i(t) \text{ with deg }r_i(t)< \text{ deg }a_i(t)$$
for some $n$. By writing $q(t)$ in this form we see that at most one of the $r_i(t)$ is nonzero: we can forget about $p(t)$ and just look at the vanishing of the coefficients of the $t^{-j}$ for $j>0$. In fact, the coefficients of $t^{-j}$ in $q_i(t) \sqrt{D(t)}$ vanish for $1 \le j \le \deg q_{i+1}-1$ and are nonzero for $t^{-\deg q_{i+1}}$. We have then $q(t)=r_n(t)q_n(t)$, $p(t)$ is forced to be $r_n(t)p_n(t)$ and $\deg a_n(t)=\deg r_n(t) + l$.
\end{proof}

\begin{remark}
By combining these lemmas, we see that we can seek the convergents of $\sqrt{D(t)}$ exactly by looking at regular functions $p(t)-uq(t)$ on the affine part of $\mathcal{C}$ such that:\\
i) the order of $p(t)-uq(t)$ at $(\infty_+)$ is more than deg $q(t)$;\\
ii) the image of their divisor in $G$ is zero.\\
Note that the convergents will be of the form $(p(t),\frac{q(t)}{D_1(t)})$.
\end{remark}

\subsection{Translation over $G$}

We now state the theorem which relates the degrees of the partial quotients and the $W_l$.\\

Let us first notice that, for $l \le g-1$, if $\delta^n \in \overline{W_l}$, then $\delta^{n+1} \in \overline{W_{l+1}}$. We prove it by induction on $g$. If $g=1$ then $W_1$ is the whole $G$. If $\delta^n \in W_l$, then $\delta^n=[P_1]+ \dots +[P_l]$ for some points $P_i \in \C \dif S$ and hence $\delta^{n+1}=[P_1] + \dots + [P_l] + [\infty_{-}]$. If $\delta^n \in \overline{W_l} \dif W_l$ we can use Corollary \ref{corollary:powersofdelta}: say, $\delta^n \in \pi_1^{-1}(\overline{W_{l-2}})$ and then $\delta^{n+1} \in \pi_1^{-1}(\overline{W_{l-1}})$ by inductive hypothesis.

\begin{thm}\label{thm:thmpartialquotwl}
Both $\delta^n \in \overline{W_l} \dif \overline{W_{l-1}}$ and $\delta^{n-1} \not \in \overline{W_{l-1}}$ occur simultaneously if and only if there exist convergents $(p(t),q(t))$ of $\sqrt{D(t)}$ such that:\\
i) $\deg$ $p(t)=n$;\\
ii) the order of $p(t)-uD_1(t)q(t)$ at $(\infty_+)$ is $\deg$ $q(t)+d-l$.
\end{thm}

\begin{proof}
We prove this by induction on $g$. For $g=1$ the properties above imply that $D(t)$ is pellian and the situation is precisely that of Remark $\ref{rmk:pelliandegree}$.\\

Step I: we first prove that if we have convergents satisfying the properties above, then $\delta^n \in \overline{W_l}$.\\

Suppose $(p(t),q(t))$ are convergents with such properties. Then:
\begin{equation*}
\text{ div}(p(t)-uD_1(t)q(t))=(\text{deg }q(t) + d-l)(\infty_+)-k(\infty_-)+\sum_{i=1}^{k-\text{deg }q(t)-d+l} (x_i)
\end{equation*}

But $k=n$ by summing with its conjugate and $\deg p(t)=\deg q(t)+d$ so the latter equality becomes:
\begin{equation*}
\text{ div}(p(t)-uD_1(t)q(t))=(n-l)(\infty_+)-n(\infty_-)+\sum_{i=1}^{l} (x_i)
\end{equation*}

Suppose that none of the $x_i$ belongs to $S$. Then, transposing this equality in $G$, the left hand side is zero by Lemma \ref{lem:pre1}, so we have:
\begin{equation*}
\delta^n = [x_1]+\dotsb+[x_l] \in W_l
\end{equation*}

If some $x_i \in S$, say $x_1=\xi_1^+$ (both $+$ and $-$ could occur), then $p(\rho_1)=0$, since $D_1(\rho_1)=0$. This implies that $\left(\frac{p(t)}{(t-\rho_1)},q(t)\right)$ are convergents for $\frac{\sqrt{D(t)}}{{t-\rho_1}}$, since they are relative prime and provide vanishing at $\infty_+$ of order $\deg$ $q(t)+d-l+1$. But then $\delta^{(1)n-1} \in W^{(1)}_{l-2}$ and by Proposition \ref{prop:closurewl}, then $\delta^n \in \overline{W_l}$.\\

Step II: we prove that if $\delta^n \in \overline{W_l}$, then we can either find suitable convergents or we have that $\delta^n \in \overline{W_{l-1}}$ or $\delta^{n-1} \in \overline{W_{l-1}}$.\\

Let $\delta^n \in \overline{W_l}$ and suppose that $\delta^n \in W_l$, so that $\delta^n=[x_1]+\dotsb+[x_l]$ with the $x_i \in \Cv \dif S$. We take the rational function with divisor
\begin{equation*}
(n-l)(\infty_+)-n(\infty_-)+\sum_{i=1}^{l} (x_i)
\end{equation*}

(this can be done because we are in the genus zero case). If $\delta^{n-1} \not \in \in \overline{W_{l-1}}$ and $\delta^n \not \in \overline{W_{l-1}}$, then none of the $x_i$ is neither $\infty_+$ nor $\infty_-$. Such function is regular on the affine part of $\mathcal{C}$ and its divisor maps to zero in $G$, so, using the Lemma \ref{lem:pre1}, it is of the form $p(t)-uD_1(t)q(t)$. As before, by considering the conjugate, we observe that $\deg p(t)=n$, $\deg q(t)=n-d$ and $p(t)-uD_1(t)q(t)$ vanishes at $\infty_+$ with order $\deg q(t)+d-l$. By Lemma \ref{lem:pre3}, $(p(t),q(t))$ are of the form $(P(t)r(t),Q(t)r(t))$ for convergents $(P(t),Q(t))$ and some polynomial $r(t)$ of degree $h$. But then the existence of the convergents $(P(t),Q(t))$ implies that $\delta^{n-h} \in \overline{W_{l-h}}$. If $h \ge 1$, then $\delta^{n-1} \in \overline{W_{l-1}}$. If $h=0$, then $(p(t),q(t))$ are the required convergents.\\

We now suppose that $\delta^n \in \overline{W_l} \dif W_l$, hence (say) $\delta^{(1)n-1} \in \overline{W^{(1)}_{l-2}}$ by Corollary \ref{corollary:powersofdelta}. Then either $\delta^{(1)n-2} \in \overline{W^{(1)}_{l-3}}$, or $\delta^{(1)n-1} \in \overline{W^{(1)}_{l-3}}$ or none of these. They imply respectively $\delta^{n-1} \in \overline{W_{l-1}}$, $\delta^n \in \overline{W_{l-1}}$ and if none of the above occurs, by induction, then there are convergents $(p_1(t),q_1(t))$ for $\frac{\sqrt{D(t)}}{t-\rho_1}$ with $\deg p_1(t)=n-1$ and which vanishes with order $\deg q_1(t)+d^{(1)}-(l-2)=\deg q_1(t)+d-l+1$. Then $(p_1(t)(t-\rho_1),q_1(t))$ are convergents for $\sqrt{D(t)}$: the vanishing holds with order $\deg q_1(t)+d-l$, the degree of $p_1(t)$ is $n$. It might be that they are not coprime: then $q_1(\rho_1)=0$. Then the correct convergents would be $\left(p_1(t),\frac{q_1(t)}{t-\rho_1}\right)$, but then, by the first part of the proof, we have that $\delta^{n-1} \in \overline{W_{l-2}}$ and in particular $\delta^{n-1} \in \overline{W_{l-1}}$.\\

Step III: we combine Step I and Step II to prove that, in the situation of Step I, actually $\delta^n \not \in \overline{W_{l-1}}$ and $\delta^{n-1} \not \in \overline{W_{l-1}}$.\\

Suppose that $\delta^n \in \overline{W_{l-1}}$. Then we can find convergents $(P(t),Q(t))$ for $\sqrt{D(t)}$ such that $\deg P(t)=n, \deg Q(t)=n-d$ and $P(t)-Q(t) \sqrt{D(t)}$ vanishes with order $n-l+1$. But then $(p(t),q(t))$ would not be convergents: they should be equal to $(P(t),Q(t))$ because they have the same degree, but they give a different order of vanishing.\\

Suppose that $\delta^{n-1} \in \overline{W_{l-1}}$. Then we can find convergents $(P(t),Q(t))$ for $\sqrt{D(t)}$ such that $\deg P(t)=n-1, \deg Q(t)=n-1-d$ and $P(t)-Q(t) \sqrt{D(t)}$ vanishes with order $n-1-l+1=n-l$. In particular, the partial quotient corresponding to $(P(t),Q(t))$ has degree $d+1-l \ge 2$, hence there can not be a partial quotient $(p(t),q(t))$ where $\deg p(t)=n$.

\end{proof}

\begin{remark}
It is crucial to us that $\mathcal{C}$ has genus zero. If the genus is $t>0$, we would expect an "error term" of magnitude $t$ for the order, since any divisor differs from a principal divisor by a suitable term of the form $[y_1]+\dotsb+[y_t]$.
\end{remark}

We can now recover many informations about partial quotients. For example, infinitely many powers of $\delta$ belong to $\overline{W_l}$ if and only if for infinitely many $n$ we have deg $a_n(t) \ge d-l$. In particular, if we know all the couples $(n,l)$ for which $\delta^n \in \overline{W_l}$, then we know all such degrees.\\

We show this with an example. Let us write the sequence $l_n$ where $l_n$ is the minimum natural number such that $\delta^n \in \overline{W_{l_n}}$. Such a sequence will be of the form:
\begin{equation*}
0,1, \dots ,g,g-2,g-1,g,g,g,g-1,g,g-3,g-2,g-1,g,g, \dots
\end{equation*}

And we will have partial quotients of degree $\deg a_0(t)=d$ (this always occurs) and then $\deg a_1(t)=3,\deg a_2(t)=1,\deg a_3(t)=1,\deg a_4(t)=2,\deg a_5(t)=4,a_6(t)=1, \dots$. Such a sequence increases by $1$ until it reaches $g$, then it drops by a natural numbers. Such drops correspond to the degrees of the partial quotients.

\subsection{Results on the degrees of the partial quotients}

We have already seen that $D(t)$ is pellian if and only if all the $b_i$ are roots of unity. We can now relate the number of $b_i$ which are roots of unity to the maximum number $h$ such that infinitely many partial quotients of $\sqrt{D(t)}$ have degree $h$. This is the same as asking whether infinitely many powers of $\delta$ lie in $\overline{W_{d-h}}$. We will use a Skolem-Mahler-Lech type theorem, that we give here in a formulation of Zannier:

\begin{thm}\cite{zannier}*{Theorem 3.2}\label{thm:smlz}
Let $\delta \in \Gm^g$. Given an increasing sequence of positive integers $c_1,c_2,\dotsc$ the Zariski-closure of $i(\delta^{c_i})$ is a finite union of points and cosets of algebraic subgroups of $\mathbb{G}_m^g$ of positive dimension.
\end{thm}

\begin{example}
Skolem-Mahler-Lech theorem is expressed in its simplest form in $\mathbb{G}_m^g$. For instance, if $g=3$ and $\lambda=(2,3,5)$, asking whether infinitely many powers of $\lambda$ belong to the variety defined by $xy-z-1=0$ is equivalent to ask whether the linear recursion $c_n=6^n-5^n-1$ has infinitely many zeroes.
\end{example}

Therefore, it is a necessary condition for $\overline{W_l}$ to contain infinitely many powers of $\delta$ that $\overline {i(W_l)}$ contains a coset of an algebraic subgroup of $\mathbb{G}_m^g$ of positive dimension. We suppose $D(t)$ not to be Pellian and we first focus on infinitely many powers of $\delta$ lying in $W_l$ (we will deal with $\overline{W_l}$ using Proposition \ref{prop:closurewl}).\\

Hence, we take the closure of the infinitely many powers of $\delta$ lying in $W_l$ and we obtain a coset $T'$ of an algebraic group. The coordinates for which $T'$ has only finitely many values correspond exactly to the $b_i$ which are roots of unity. We can take an irreducible algebraic subgroup $T \sub T'$ of dimension one, moreover keeping the nonconstant coordinates to be nonconstant (just intersect $T$ with $x_1= \dotsc =x_j$; the algebraic subgroups of $\Gm^g$ are given by systems of equations of the form $x_1^{a_1} \dotsm x_g^{a_g}=1$). This attention to nonconstant coordinates will be of importance later, when we will use the projection $\pi_j: G \rightarrow G^{(j)}$.\\

Thus, we have a translate $T$ of a one-dimensional algebraic group contained in $\overline{W_l}$ and we assume for now that the intersection with $W_l$ is infinite. Recalling that we have a map $\phi: (\mathcal{C} \setminus S)^l \rightarrow G \cong \mathbb{G}_m^g$, then the preimage of $T$ in $(C \dif S)^l$ is infinite, so we can take an irreducible curve $E$ contained in $\phi^{-1}(T)$ whose image via $\phi$ is not constant. We can then lift the coordinates $x_i$ on $T \sub G \cong \Gm^g$ on $E$ and we can write down explicitely the $\phi^*x_i=x_i(\phi)$. If the $z_i$ are the coordinates of $E$ with respect to $(\mathcal{C} \setminus S)^g$, then $\phi^*x_i$ is the $i$-th component of $\phi$, so that
\begin{equation*}
\phi^*x_i=\prod_{j=1}^l \frac{z_j-b_i}{z_j-b_i^{-1}}
\end{equation*}

By choosing a nonsingular complete model $X$ for $E$, we have that (since $E$ is a curve) these functions lift to rational functions on $X$ (none of them is equal to the constants $b_i^{\pm 1}$, since $E \sub (\Cv \dif S)^g$ corresponds to an open set of $X$). These functions generate a multiplicative group of rank $\le 1$ on $E$, since the $x_i$ do and so do the $\phi^*x_i$. The rank is not zero since $\phi(E)$ is not a point by our choice of $E$. The corresponding functions on $X$ have the same property.\\

We can now prove that this doesn't happen often:\\

\begin{lem}\label{lem:curveunity}
Let $X$ be a complete nonsingular curve with functions $z_1,\dotsc,z_l$. Suppose that
\begin{itemize}
    \item none of the $z_j$ is equal to one of the constants $b_i^{\pm 1}$
    \item the functions 
    $$f_i=\left( \prod_{j=1}^l \frac{z_j-b_i}{z_j-b_i^{-1}} \right)\qquad i=1,\dotsc,g$$ form a multiplicative group of rank $1$ (modulo constants, i.e. modulo $\C^*$).
\end{itemize}
then
\begin{enumerate}
    \item at most $l$ of the $f_i$ are nonconstant;
    \item at most $l-1$ of the $f_i$ are constant.
\end{enumerate}
\end{lem} 
\begin{proof}
Notice that the zeroes of the function $\frac{z_j-b_i}{z_j-b_i^{-1}}$ are the zeroes of $z_j-b_i$ and its poles are the zeroes of $z_j-b_i^{-1}$, because $z_j-b_i$ and $z_j-b_i^{-1}$ have the same poles.

Part 1) If, say, $f_1,\dotsc,f_{l+1}$ are nonconstant, then by the condition on the rank there is a point $P$ which is a zero or a pole of each of them.
By the pigeonhole principle, this implies that, for some $1\leq j\leq l$ and two different $1\leq i,k\leq l+1$, the point $P$ is both a zero of $z_j-b_i$ (or $z_j-b_i^{-1}$) and of $z_j-b_k$ (or $z_j-b_k^{-1}$), but this contradicts our assumptions on the $b_i$.

\medskip

Part 2) We proceed by a double-counting argument. We want upper and lower estimates of:
$$\sum_{j<k} \text{deg} (z_jz_k-1)$$
First of all, if $z_jz_k-1=0$ then $z_k=z_j^{-1}$ and:
$$\frac{z_j-b_i}{z_j-b_i^{-1}} \cdot \frac{z_k-b_i}{z_k-b_i^{-1}}=\frac{z_j-b_i}{z_j-b_i^{-1}} \cdot \frac{z_j^{-1}-b_i}{z_j^{-1}-b_i^{-1}}=b_i^2$$
So we can simply remove $z_j$ and $z_k$ from our set of functions and conclude by an inductive argument. If $z_jz_k-1$ is always nonzero we can then estimate:
$$\sum_{j<k} \deg (z_jz_k-1) \le \sum_{j<k}\left( \deg z_j + \deg z_k\right)=(l-1) \sum_{j=1}^l \deg z_j$$
since $\text{deg} (z_jz_k-1) = \text{deg} (z_jz_k) \le \text{deg } z_j + \text{deg } z_k$.\\

Let now $P$ be a zero of both $z_j-b_i$ and $z_k-b_i^{-1}$. Then by
$$z_jz_k-1=(z_j-b_i)(z_k-b_i^{-1})+b_i(z_k-b_i^{-1})+b_i^{-1}(z_j-b_i)$$
we have ord$_P(z_jz_k-1) \ge $ min$\{$ord$_P(z_j-b_i)$, ord$_P(z_k-b_i^{-1}) \}$.\\
Suppose that, say, $f_1,\dotsc,f_l$ are constant. We choose some $i \le l$ and some point $P \in X$ such that $P$ is a zero of some of the $z_j-b_i^{\pm 1}$. Let $S_i$ be the subset of $\{ 1,\dotsc,l \}$ such that $j \in S_i \leftrightarrow z_j(P) = b_i^{\pm 1}$. Notice that  the $S_i$ are disjoint. We now choose $h \in S_i$ to be such that ord$_P (z_h-b_i^{\pm 1})$ is maximal. Assume for simplicity that the exponent is $-1$. Then:
$$\sum_{j \in S_i \setminus \{ h \}} \text{ord}_P(z_jz_h-1) \ge \sum_{j \in S_i \setminus \{ h \}} \text{min} \{ \text{ord}_P(z_h-b_i^{-1}), \text{ ord}_P(z_j-b_i) \} = \sum_{j \in S_i} \text{ ord}_P(z_j-b_i) = \sum_{j \in S_i} \text{ ord}_P(z_j-b_i^{-1}) $$
where the last equality follows from the fact that $f_i$ is constant. We have then:
$$2 \sum_{j < k, j,k \in S_i} \text{ord}_P(z_jz_k-1) \ge 2 \sum_{j \in S_i \setminus \{ h \}} \text{ord}_P(z_jz_h-1) \ge \sum_{j \in S_i} \text{ ord}_P(z_j-b_i) + \sum_{j \in S_i} \text{ ord}_P(z_j-b_i^{-1})$$
Keeping $P$ fixed, we sum over $i$:
$$2\sum_{j<k} \text{max} \{ \text{ord}_P(z_jz_k-1), 0 \} \ge 2 \sum_{i=1}^l \sum_{\substack{j < k \\ j,k \in S_i}} \text{ord}_P(z_jz_k-1) \ge \sum_{i=1}^l \left( \sum_{j \in S_i} \text{ ord}_P(z_j-b_i) + \sum_{j \in S_i} \text{ ord}_P(z_j-b_i^{-1}) \right)$$
We sum now over $P \in X$, noting that the left hand side becomes the sum of degrees we were looking for:
$$2\sum_{j<k} \text{deg} (z_jz_k-1) \ge \sum_{i=1}^l \sum_{P \in X} \left( \sum_{j \in S_i} \text{ ord}_P(z_j-b_i) + \sum_{j \in S_i} \text{ ord}_P(z_j-b_i^{-1}) \right)=\sum_{i=1}^l \sum_{j=1}^l \left( \deg (z_j-b_i) + \deg (z_j-b_i^{-1})\right)$$
where the last equality follows from the definition of $S_i$ (which depends on $P$). We have then
$$\sum_{j<k} \text{deg} (z_jz_k-1) \ge l \sum_{j=1}^l \text{deg } z_j$$
which, combined with our previous estimate, provides a contradiction.
\end{proof} 

\begin{remark}
It is not very clear to us why one should consider the functions $z_jz_k-1$ and what they do represent, but they encode the cancellation of numerators and denominators keeping a reasonable simmetry and simplicity.
\end{remark}

\begin{thm}\label{thm:rootsofunity}
We assume $b_1^{\pm 1}, \dotsc , b_g^{\pm 1}$ to be distinct. Let $r$ be the number of roots of unity among the $b_i$. Then:

\begin{enumerate}
\item Only finitely many partial quotients of $\sqrt{D(t)}$ have degree $\ge r+2$;
\item Let the roots of unity be precisely $b_1, \dotsc , b_r$ and let $D_p(t)=(t^2-1)(t-\rho_1)^2 \dotsm (t-\rho_r)^2$. Then $\sqrt{D(t)}$ has only finitely many partial quotients of degree $\ge d-r$, with the exceptions of the partial quotients $a_n$ corresponding to those indices $n$ such that the convergents $(p_n(t), q_n(t))$ are of the form $(p(t)(t-\rho_{r+1}) \dotsm (t-\rho_g),q(t))$, where $(p(t),q(t))$ are convergents of $\sqrt{D_p(t)}$. In other words, up to finitely many exceptions, every partial quotient with degree $\ge d-r$ is obtained from  $\sqrt{D_p(t)}$.
\end{enumerate}

\end{thm}

\begin{remark}
If $D(t)$ is Pellian then (1) is a consequence of Remark \ref{rmk:pelliandegree} and (2) holds vacuously, since $\sqrt{D(t)}$ equals its Pellian part.
\end{remark}

\begin{proof}
(1): Suppose the converse. By Theorem \ref{thm:thmpartialquotwl} infinitely many powers of $\delta$ lie in $\overline{W_{g-r-1}}$. We can construct as above a one-dimensional translate of an algebraic subgroup $T$, contained in $\overline{W_{g-r-1}}$, whose constant coordinates correspond exactly to the roots of unity among the $b_i$.\\

We prove by induction on $g$ that if $\overline{W_l}$ contains a positive dimensional translate $T$ of an algebraic group whose constant coordinates correspond to the $b_i$ which are roots of unity, then $r \ge g-l$. In our case, this proves (1): we have $r \ge g-(g-r-1)=r+1$.\\

If $g=1$, then the assertion holds obviously: if $l=0$ then $\overline{W_0}$ is a point and if $l=1$ then $r \ge 0$.\\

We now take such a $T$; if it has an infinite intersection with $W_l$, then we can apply Lemma \ref{lem:curveunity}, (1): the number of constant coordinates for $T$ is at least $g-l$, proving our assertion.\\

Suppose now that $T$ has finite intersection with $W_l$. Then $T$ has infinite intersection with, say, $\delta \pi_1^{-1}\left(\overline{W^{(1)}_{l-2}}\right)$. By writing $T^{(1)}=\pi_1(T)$, we observe that $r^{(1)}$ is either $r$ or $r-1$ (depending on whether $b_1$ is a root of unity).
\begin{itemize}
\item It might be that $T^{(1)}$ is a point: this happens precisely if $b_1$ is the only non-root of unity. In this case, $r=g-1$ and of course $l \ge 1$, since $\overline{W_0}$ is zero dimensional, so we are done.
\item If $T^{(1)}$ is one dimensional and $b_1$ is a root of unity, then $r^{(1)}=r-1$ and $\delta^{-1}T^{(1)} \sub \overline{W^{(1)}_{l-2}}$. By inductive hypothesis $r-1=r^{(1)} \ge (g-1)-(l-2)=g-l+1$, so we are done.
\item If $T^{(1)}$ is one dimensional and $b_1$ is a not root of unity, then $r^{(1)}=r$ and $\delta^{-1}T^{(1)} \sub \overline{W^{(1)}_{l-2}}$. By inductive hypothesis $r=r^{(1)} \ge (g-1)-(l-2)=g-l+1$, so we are done.
\end{itemize}

(2): Suppose that infinitely many powers of $\delta$ lie in $\overline{W_r}$. Then we can find, as before, a one-dimensional irreducible translate of an algebraic group $T$ contained in $\overline{W_r}$, whose constant coordinates correspond to the $b_i$ which are roots of unity.\\

We denote with the superscript $(p)$ the usual quantities referred to the polynomial $D^{(p)}(t)=(t^2-1)(t-\rho_1)^2 \dotsm (t-\rho_r)^2$. Let $\pi_p: \Gm^g \cong G \rightarrow G^{(p)} \cong \Gm^r$ be the projection on the first $r$ coordinates, let $l$ be a positive integer and let $h=l-2(g-r)$. We prove by induction on $g$ that if $\overline{W_l}$ contains a one dimensional translate of an algebraic group $T$ whose constant coordinates correspond to $b_1, \dotsc , b_r$, then either $T \sub \delta^{g-r} \pi_p^{-1}\left(\overline{W^{(p)}_h}\right)$ or $l \ge r+1$.\\

We first show why this assertion implies (2). Suppose that we have convergents $(p(t),q(t))$ for $\sqrt{D(t)}$ with partial quotient $a(t)$ of degree $d-r' \ge d-r$. Then, if $n=\deg p(t)$, by Theorem \ref{thm:thmpartialquotwl} we have that $\delta^n \in \overline{W_{r'}} \dif \overline{W_{r'-1}}$ and $\delta^{n-1} \not \in \overline{W_{r'-1}}$. Let $h'=r'-2(g-r)$. Suppose that $\delta^n \in \delta^{g-r} \pi_p^{-1}\left(\overline{W^{(p)}_{h'}}\right)$. We have that $\delta^n \not \in \delta^{g-r} \pi_p^{-1}\left(\overline{W^{(p)}_{h'-1}}\right) \sub \overline{W_{r'-1}}$ and that $\delta^{n-1} \not \in \delta^{g-r} \pi_p^{-1}\left(\overline{W^{(p)}_{h'-1}}\right) \sub \overline{W_{r'-1}}$. In this situation, there are convergents $(P(t),Q(t))$ of $\sqrt{D_p(t)}$ such that $\deg P(t)=n-g+r,$ $\deg Q(t)=n-g-1$ and $P(t)-\sqrt{D^{(p)}}Q(t)$ vanishes at $(\infty_+)$ with order $\deg Q(t)+r+1-h'=n-g+r-h'$ (always Theorem \ref{thm:thmpartialquotwl}). Then the pair $(p(t),q(t))=(P(t)(t-\rho_{r+1}) \dotsm (t-\rho_g),Q(t))$ satisfies $\deg p(t)=n$, $ \deg q(t)=n-g-1$ and order of vanishing of $p(t)-q(t)\sqrt{D(t)}$ equal to $n-h'-2(g-r)=\deg q(t)-g-h'+2r=\deg q(t)+d-r'$. Therefore $(p(t),q(t))$ are of the form prescribred in (2). Notice that $p(t)$ and $q(t)$ are necessarily coprime: if $Q(\rho_i)=0$ for some $i \ge r+1$ we would get $\delta^{n-1} \in \overline{W_{r'-1}}$.\\

If there were infinitely many powers of $\delta$ such that $\delta^n \in \overline{W_{r'}} \dif \overline{W_{r'-1}}$ and $\delta^{n-1} \not \in \overline{W_{r'-1}}$ with $\delta^n \not \in \delta^{g-r} \pi_p^{-1}\left(\overline{W^{(p)}_{h'}}\right)$, we could take the closure of such powers and obtain a positive dimensional translate of an algebraic group not contained in $\delta^{g-r} \pi_p^{-1}\left(\overline{W^{(p)}_{h'}}\right)$, but contained in $\overline{W_{r'}}$, with constant coordinates as usual. We can moreover extract a one dimensional $T$ with the same properties: just choose a point not belonging to $\delta^{g-r} \pi_p^{-1}\left(\overline{W^{(p)}_{h'}}\right)$ and consider the translation by that point of a one dimensional algebraic subgroup with the property that the nonconstant coordinates are exactly corresponding to the $b_i$ which are not roots of unity. But now the assertion implies that $r' \ge r+1$, which is a contradiction.\\

We now prove the aforementioned assertion. If $g=1$ then either $r=1$, so $\delta^{g-r} \pi_p^{-1}\left(\overline{W^{(p)}_h}\right)$ is just $\overline{W_l}$, or $r=0$, so $l \ge r+1$.\\

Suppose that $T$ has an infinite intersection with $W_l$; we can apply Lemma \ref{lem:curveunity}, (2): at most $l-1$ coordinates on $T$ are constant, but we know that the constant coordinates are precisely $r$, hence $r \le l-1$.\\

If $T$ does not have an infinite intersection with $W_l$, then it has infinite intersection with, say, $\delta \pi_j^{-1}\left(\overline{W^{(j)}_{l-2}}\right)$. As before, we write $T^{(j)}=\pi_j(T)$ and $r^{(j)}$ for the number of roots of unity among the $b_i$, but excluding $b_j$. Moreover, we also consider the projection $\pi^{(j)}_p: G^{(j)} \rightarrow G^{(j)}_p$, where we project onto the first $r$ coordinates if $j \ge r+1$ and onto the first $r$ coordinates, $j$ excluded, if $j \le r$. We denote with both the superscripts $(p)$ and $(j)$ the relevant objects related to $D^{(p)}(t)$ if $j \ge r+1$ or to $\frac{D^{(p)}(t)}{t-\rho_j}$ otherwise. We set $h(j)=h$ if $j \ge r+1$ and $h(j)=h-2$ otherwise. Notice that if $T^{(j)} \sub \delta^{g-r} \pi^{(j)-1}_p \left(\overline{W^{(j)(p)}_{h(j)}}\right)$ then $T \sub \delta^{g-r} \pi^{-1}_p \left(\overline{W^{(p)}_h}\right)$. There are different possibilities:
\begin{itemize}
\item If $b_j$ is a root of unity, say $j=1$, then $\delta^{-1} T^{(1)}$ is one dimensional and contained in $W^{(1)}_{l-2}$ and $r^{(1)}=r-1$. $\delta^{-1} T^{(1)}$ is not contained in $\delta^{g-r} \pi^{(1)-1}_p \left(\overline W^{(p)}_h\right)$ and hence by inductive hypothesis we obtain $l-2 \ge r^{(1)}+1$, so we are done.
\item If $b_j$ is not a root of unity, say $j=g$, and $T^{(g)}$ is one dimensional, we have that it is contained in $W^{(g)}_{l-2}$ and $r^{(g)}=r$. $\delta^{-1}T^{(g)}$ is not contained in $\delta^{g-r} \pi^{(g)-1}_p \left(\overline{W^{(j)(p)}_{h-2}}\right)$ and by inductive hypothesis we obtain $l-2 \ge r$, so we are done.
\item If $b_j$ is not a root of unity, say $j=g$, and $T^{(g)}$ is zero dimensional, then $b_g$ is the only root of unity among the $b_i$. If $T$ is not contained in $\delta^{-1} \pi_g^{-1}\left(\overline{W^{(g)}_{l-2}}\right)$, then it has finite intersection (both are closed) and it has therefore either infinite intersection with $W_l$ (this gives the bound, as proven above) or infinite intersection with some $\delta^{-1} \pi_j^{-1} \left(\overline{W^{(j)}_{l-2}}\right)$. This implies the inequality by induction, as proven before.

\end{itemize}

\end{proof}

\begin{remark}
The argument above, when we prove (2) using the relevant assertion, proves also that if $(p(t),q(t))$ is a partial quotient of a Pellian polynomial, then $q(t)$ divides a Chebyshev polynomial. This is because what happens in $G^{(p)}$ is independent of $\rho_{r+1}, \dotsc , \rho_g$.

\end{remark}

\begin{remark}
Notice that, for any $l < \frac{d}{2}$, $\delta^n \not \in \overline{W_l}$ eventually, unless this inclusion is explained by a Pellian polynomial. This situation is compatible with a Theorem of Zannier (\cite{zannier}*{Theorem 1.3}).
\end{remark}

\subsection{Examples}\label{sec:examples}
We now show examples for which our previous bounds cannot be sharpened, in particular excluding the convergents that come from Pellian polynomials.\\
\\
We fix positive integers $\frac{d}{2} \le l \le g$, $p \le l-1$ and $q \le l$ such that $p+q=g$. Let $\omega$ be a primitive $2l$-th root of unity. Let $h$ be a non-root of unity. We define:
$$b_i=\omega^i \text{ for } 1 \le i \le p$$
$$b_i=h \cdot \omega^{2i} \text{ for } p+1 \le i \le p+q$$
We claim that the corresponding $\delta$ satisfies $\delta^n \in W_l$ for infinitely many $n$. We check this via the isomorphism $i$:
$$\delta^n \in W_l \text{ if and only if there exist } z_j \in \mathcal{C} \setminus S \text{ such that } \prod_{j=1}^l \frac{z_j-b_i}{z_j-b_i^{-1}}=b_i^{2n} \text{ for every $1 \le i \le g$}$$
We now put:
$$n=2lk \text{ with $k$ a positive integer}$$
$$z_j=t \omega^{2j} \text{ for $t$ so that no $z_j$ is a $b_i^{\pm 1}$}$$
If $1 \le i \le p$ we get:
$$\prod_{j=1}^l \frac{z_j-b_i}{z_j-b_i^{-1}}=\prod_{j=1}^l \frac{t \omega^{2j}-\omega^i}{t \omega^{2j}-\omega^{-i}}=\prod_{j=1}^l \omega^{2i} \frac{t \omega^{2j}-\omega^i}{t \omega^{2j+2i}-\omega^i}=1=b_i^{2lk}$$
If $p+1 \le i \le p+q$ we get:
$$\prod_{j=1}^l \frac{z_j-b_i}{z_j-b_i^{-1}}=\prod_{j=1}^l \frac{t\omega^{2j}-h\omega^{2i}}{t\omega^{2j}-h^{-1}\omega^{-2i}}=\prod_{j=1}^l \omega^{4i} \frac{t\omega^{2j-2i}-h}{t\omega^{2j+2i}-h^{-1}}=\prod_{j=1}^l \frac{t\omega^{2j}-h}{t\omega^{2j}-h^{-1}}$$
and $b_i^{2lk}=h^{2lk}$, which are both independent of $i$. The rational map:
$$f: \mathbb{G}_m \rightarrow \mathbb{P}_1$$
$$f(t)=\prod_{j=1}^l \frac{t\omega^{2j}-h}{t\omega^{2j}-h^{-1}}$$
is nonconstant (for instance, $f(h)=0, f(h^{-1})=\infty$), so its image (a constructible) is dense in $\mathbb{P}_1$, therefore it contains $h^{2lk}$ for $k$ big enough, so we're done.

\begin{remark}
$p$ is the number of roots of unity among the $b_i$. By Theorem \ref{thm:rootsofunity}, $p+1 \le l$ and $q = g-p \le l$, so that our examples show that any $p$ such that $g-l \le p \le l-1$ can be realized.
\end{remark}

These examples shows that there are polynomials $D(t)$ with a non-periodic expansion in which infinitely many partial quotients have a degree greater than one.\\

For instance, if we consider $g=4$, $\zeta$ a primitive sixth root of unity and \[b_1=\zeta,b_2=\zeta^2,b_3=2\zeta^2,b_4=2\zeta^4,\] so that \[\rho_1=\frac{1}{2}, \rho_2=-\frac{1}{2}, \rho_3=\frac{-5+3\sqrt{3}i}{8},\rho_4=\frac{-5-3\sqrt{3}i}{8},\] we obtain the polynomial:
\begin{align*}
D(t)&=(t^2-1)\left(t-\frac{1}{2}\right)^2\left(t+\frac{1}{2}\right)^2(t-\left(\zeta^2+\frac{\zeta^{-2}}{4}\right))^2(t-\left(\zeta^{-2}+\frac{\zeta^2}{4}\right))^2\\
&=\frac{1}{4096}(t^2-1)(4t^2-1)^2(16t^2+20t+13)^2.
\end{align*}

We claim that the partial quotients of $\sqrt{D(t)}$ have degree $5,1,2,1,2, \dotsc$. The first partial quotient corresponds to the polynomial part. We also see that $\delta^{3n}$ belongs to $W_3$ (as it has been explicitely computed above: in this case we have $p=q=2$ and $h=2$). No power of $\delta$ belongs to $\overline{W_2}$ with the exception of $\delta$ and $\delta^2$. If $\delta^n \in \overline{W_2} \dif W_2$ then $\delta^{n-1} \in \pi_j^{-1}(W^{(j)}_0)$ for some $j=1,2,3,4$, but this is false. We compute:
\begin{equation*}
\frac{z_1-\zeta}{z_1-\zeta^5} \frac{z_2-\zeta}{z_2-\zeta^5}=\zeta^{2n} \leftrightarrow z_1z_2(1-\zeta^{2n})-(z_1+z_2)(\zeta-\zeta^{2n+5})+(\zeta^2-\zeta^{2n+4})=0
\end{equation*}
\begin{equation*}
\frac{z_1-\zeta^2}{z_1-\zeta^4} \frac{z_2-\zeta^2}{z_2-\zeta^4}=\zeta^{4n} \leftrightarrow z_1z_2(1-\zeta^{4n})-(z_1+z_2)(\zeta^2-\zeta^{4n+4})+(\zeta^4-\zeta^{4n+2})=0
\end{equation*}

If $n \equiv 0 \pmod 3$ then we obtain either $-1=z_1+z_2=1$ or both $z_1,z_2=\infty$, for which $\frac{z_1-b_3}{z_1-b_3^{-1}}\frac{z_2-b_3}{z_2-b_3^{-1}} \neq b_3^{2n}$. If $n \equiv 1 \pmod 3$ we obtain $=z_1z_2=1$, hence $\frac{z_1-b_3}{z_1-b_3^{-1}}\frac{z_2-b_3}{z_2-b_3^{-1}}=b_3^2$, or $z_1=\infty$ and $z_2=0$ (or viceversa) and the same holds. If $n \equiv 2 \pmod 3$ then $z_1z_2=z_1+z_2=0$, for which $\frac{z_1-b_3}{z_1-b_3^{-1}}\frac{z_2-b_3}{z_2-b_3^{-1}}=b_3^4$, or $z_1=\infty$ and $1=z_2=-1$ (or viceversa). In any case, if $n \ge 3$ then $\delta^n \not \in \overline{W_2}$.\\

Therefore the degrees of the partial quotients (with the exception of the first) can only be $1$ or $2$ and, since $\delta^{3n} \in \overline{W_3}$ for every positive integer $n$, the only admissible sequence is $5,1,2,1,2 \dotsc$.

\subsubsection{If $g=2$}
Say that $D(t)=(t^2-1)(t-\rho_1)^2(t-\rho_2)^2$ with $\rho_1\neq \rho_2$ and $\rho_1,\rho_2\neq\pm 1$.

In this case the  closure of $W_1$ is given by the equation
\[
(b_1^{-1}-b_2^{-1})x_1x_2 +(b_2-b_1^{-1})x_1 +(b_2^{-1}-b_1)x_2+b_1-b_2=0.
\]
According to how many of the $b_i$ are roots of unity we have the following cases:
If both $b_1,b_2$ are roots of unity, then $D(t)$ is Pellian.

If $b_1$ is a root of unity, and $b_1^2$ has order $m$, then $\deg a_n$ is eventually 1, and one convergent of $\sqrt{D(t)}$ every $m$ is given by a solution to the Pell equation $x(t)-(t^2-1)(t-\rho_1)^2 y(t)^2=1$, which is solvable because $(t^2-1)(t-\rho_1)^2$ is Pellian.

If none of $b_1,b_2$ is a root of unity, then again $\deg a_n$ is eventually 1.

\subsubsection{If $g=4$}

We now exhibit a family of polynomials which admit infinitely many partial quotients of degree $2$ which is not of the form above.\\

We set $b_1=\theta_1, b_2=h, b_3=\theta_2 h, b_4=\theta_3 h$, where the $\theta_i$ are roots of unity of order odd $N_i$. We consider all the $n$ such that $2n-3$ is a multiple of all the $N_i$ and we take the closure of the set consisting in all the $\delta^n$. Such a set $Z$ is parametrized by $(\theta_1^3,t,\theta_2^3t,\theta_3^4t)$ for $t \in \C^*$. Computations imply that if:
\begin{equation*}
h^3(1+\theta_1+\theta_1^2)\theta_2\theta_3-h^2(\theta_2+\theta_3+\theta_2\theta_3)\theta_1-h(1+\theta_2+\theta_3)\theta_1+(1+\theta_1+\theta_1^2)=0
\end{equation*}
Then $Z \sub \overline{W_3}$. Then we can choose $\theta_1,\theta_2,\theta_3$ as we wish and find $h$ which satisfies the equation above. This family of examples involves roots of unity of arbitrarily large order with a fixed $g$.

\section{Non squarefree case: Notation and setting}\label{sec:nonsquarefreeGenJac}
We now review our constructions in the case of a non squarefree $D_1(t)$.\\

As above, let $\Cv:\{U^2=T^2-V^2\} \subseteq \P_2$; this is a smooth curve of genus 0. We fix the affine chart given by $V \neq 0$ and use affine coordinates $t=T/V$ and $u=U/V$. With respect to this chart $\Cv$ has two points at infinity,  $(\infty_+)=(1:1:0)$ and $(\infty_-)=(-1:1:0)$.\\

Let now $D_1(t)=(t-\rho_1)^{e_1} \dotsm (t-\rho_g)^{e_g} \in  \C[t]$ be a monic polynomial of degree $h$ with all the $\rho_i$ distinct and such that $\rho_i \neq \pm 1$. Let $D(t)=(t^2-1)D_1(t)^2$, which has degree $2s$.

For each $\rho_i$ we define as before $\xi_i^\pm=(\rho_i,\pm\sqrt{\rho_i^2-1}) \in \Cv$ for some choice of the square roots. We define $\emme=\sum_{i=1}^g e_i\left((\xi_i^+) + (\xi_i^-) \right)$, which is a divisor on $\Cv$, and we denote by $S$ its support. We take $\emme$ as the modulus and write $J_\emme$ for the generalized Jacobian of $\left(\mathcal{C},\emme\right)$.\\

Let $\delta$ be the class of the divisor $(\infty_-)-(\infty_+)$ in $J_\emme$. In the non squarefree case there are no Pellian polynomials and Proposition \ref{prop:pellian-gen} never occurs. In fact, all the convergents $(p(t),q(t))$ for a Pellian polynomial have the property that $q(t)D_1(t)$ must divide a Chebyshev polynomial, as explained in the Introduction; these are all squarefree.\\

In order to deal with the multiplicities, we will need to use derivatives. Let $s=t+u$ (notice that $\C(\Cv)=\C(s)$) and we define $\dplus=s{\diff{}{s}}$ and $\dminus=s^{-1}{\diff{}{s^{-1}}}$. We denote by $\dplus^k$ and $\dminus^k$ the $k$-th iterate of $\dplus$ and $\dminus$ respectively. These derivations are well behaved with respect to the Galois conjugation $g: (t,u) \rightarrow (t,-u)$. Any rational function on $\Cv$ can be written uniquely as $p(s+\frac{1}{s})+(s-\frac{1}{s})q(s+\frac{1}{s})$, since $2t=s+\frac{1}{s}$ and $2u=s-\frac{1}{s}$. If we apply $\dplus$:
\begin{equation*}
\dplus\left(p\left(s+\frac{1}{s}\right)+\left(s-\frac{1}{s}\right)q\left(s+\frac{1}{s}\right)\right)=\left(s-\frac{1}{s}\right)p'\left(s+\frac{1}{s}\right)+\left(s-\frac{1}{s}\right)^2q'\left(s+\frac{1}{s}\right)+\left(s+\frac{1}{s}\right)q\left(s+\frac{1}{s}\right)
\end{equation*}

And we notice that $g$ interchanges $s$ and $s^{-1}$, so that $\dplus \circ g + g \circ \dplus=0$. In particular, if a function $f$ depends on $t$ only, then $\dplus f$ is $u$ times a function which depends on $t$ only. The same holds for $\dminus$ (just think of $s^{-1}$ instead of $s$).\\

We work again with a quotient of $J_\emme$. $\Gal(\C(\Cv)/\C(t))$ acts on the points of $\mathcal{C}$ and by linearity on $\Div(\Cv)$ and on $J_\emme$. Let $\hat{J}_\emme$ be the subgroup of $J_\emme$ invariant for the action of $g$, and define $G=J_\emme/\hat{J}_\emme$.\\

We will now construct an explicit isomorphism between $G$ and $\Gm^g \times \Ga^{\sum_{i=1}^g (e_i-1)}$ (we will use multiplicative notation). Let $E\in\Div^0_\emme(\Cv)$ be a divisor of degree zero. Since $\mathcal{C}$ has genus zero, we can write $E=\div(f)$ for some $f\in\C(\Cv)^*$, and we define
\begin{equation*}
i(E)=\left(\dotsc,\frac{f(\xi_i^+)}{f(\xi_i^-)}, \dotsc , \frac{\dplus f}{f}(\xi^+_i)-\frac{\dminus f}{f}(\xi^-_i), \dotsc , \dplus^{k} \left(\frac{\dplus f}{f}\right) (\xi^+_i)-(-1)^k\dminus^{k}\left(\frac{\dminus f}{f}\right)(\xi^-_i), \dotsc \right)
\end{equation*}

For all $k \le e_i-2$. In other words, we are taking the logarithmic derivative and its derivatives up to order $e_i-1$ for each $\rho_i$. We see that $i:\Div^0_\emme(\Cv)\to \Gm^g$ is a well-defined group homomorphism, since $\dplus$ and $\dminus$ are linear are linear and for every derivation $\frac{(fg)'}{fg}=\frac{f'}{f}+\frac{g'}{g}$ holds.

\begin{prop}\label{prop:smallgeneralizedjacobiannonsqr}
The map $i:\Div^0_\emme(\Cv)\to \Gm^g$ induces a group isomorphism $\iota:G\to\Gm^g \times \Ga^{\sum_{i=1}^g (e_i-1)}$.
\end{prop}

\begin{proof}
First, we notice that if $f$ depends on $t$ only, then by the argument explained above $\dplus^{k}\left(\frac{\dplus f}{f}\right)$ does as well if $k$ is odd and it does after a multiplication by $u$ if $k$ is even. Since $g \dplus^{k}\left(\frac{\dplus f}{f}\right)(\xi_i^+)=\dminus^{k}\left(\frac{\dminus f}{f}\right)(\xi_i^-)$ (we are just interchanging $u$ and $-u$), then $i(\div f)=0$. Moreover if $1-f$ vanishes at every $\xi_i^{\pm}$ with order $e_i$, then all the derivatives above vanish: an explicit computation shows that the derivatives of $f'/f$ only admit $f$ at the denominator and in the numerators only derivatives up to the $e_i-1$ occur while each monomials contains a derivative of order at least one. We conclude that $\iota$ is well-defined.\\

We check injectivity. Let $f$ be a rational function such that $\div(f)\in\ker \iota$, so that
$$f(\xi_{\rho_i}^+)=f(\xi_{\rho_i}^-) \quad \forall i=1,\dotsc,g.$$
$$\dplus^k\frac{\dplus f}{f}(\xi_i^+)=(-1)^k\dminus^{k}\frac{\dminus f}{f}(\xi_i^-) \quad \forall i=1,\dotsc,g \text{ and } k=0, \dotsc, e_i-2$$
We can choose a Galois-invariant function $p(t)\in\C(t)^* \sub \C(\Cv)^*$ that coincides with $f$ on all the $\xi_{\rho_i}^{\pm}$ and it coincides also with the derivatives. In order to see this, let us compute the first values of $\dplus^k \frac{\dplus p(s+\frac{1}{s})}{p(s+\frac{1}{s})}$:
\begin{equation*}
\frac{\left(s-\frac{1}{s}\right)p'\left(s+\frac{1}{s}\right)}{p\left(s+\frac{1}{s}\right)}
\end{equation*}
\begin{equation*}
\frac{\left(s+\frac{1}{s}\right)p'\left(s+\frac{1}{s}\right)p\left(s+\frac{1}{s}\right)+\left(s-\frac{1}{s}\right)^2p''\left(s+\frac{1}{s}\right)p\left(s+\frac{1}{s}\right)+\left(s-\frac{1}{s}\right)^2p'\left(s+\frac{1}{s}\right)^2}{p\left(s+\frac{1}{s}\right)^2}
\end{equation*}

And we observe that the denominator is always $p(s+\frac{1}{s})^k$, while the unique coefficient of the $k$-the derivative of $p$ is $(s-\frac{1}{s})^kp\left(s+\frac{1}{s}\right)^{k-1}$ (this is easily seen by induction). In particular it is nonzero at the $\xi_i^{\pm}$, since the $\rho_i$ are different from $\pm 1$. We can then interpolate: for each $\xi_i^+$ we just need to prescribe the value and the derivatives of $p$ at some points. The values for the $\xi_i^-$ are automatically correct by the condition above. Therefore (as the logarithmic derivative gives an homomorphism to $\Ga$) $\div(\frac{f}{p})$ is zero in $J_\emme$. Since $\div(p)$ is zero in $G$ we have that also $\div(f)$ is zero in $G$; this shows the injectivity.\\

We check that $i$ is surjective and this implies the surjectivity of $\iota$. Let $p(s)$ be a polynomial which vanishes at each $b_i^{\pm 1}=s(\xi_i^{\pm})$. Let us compute the first values of $\dplus^k \frac{\dplus 1+p(s)}{1+p(s)}$:
\begin{equation*}
\frac{sp'(s)}{1+p(s)}
\end{equation*}
\begin{equation*}
\frac{sp'(s)+sp(s)p'(s)+s^2(1+p(s))p''(s)+s^2p'(s)^2}{(1+p(s))^2}
\end{equation*}
And we observe that the denominator is $(1+p(s))^k$ and the coefficient of $p''(s)$ is $s^k(1+p(s))^{k-1}$. We can choose $p(s)$ (which vanishes at the $b_i^{\pm 1}$) in order to obtain the whole $\Ga$ part and being $1$ in the $\Gm$ part, by prescribing each derivative. Then we can fix the $\Gm$ part by interpolation (i.e. we first choose a function that gives the correct values at the $\Gm$ part and then we choose the polynomial $p(s)$).

\end{proof}

Now we seek the image of $\delta$ in $G$, which is the image of $(\infty_-)-(\infty_+)$. We see that:
\begin{equation*}
\div(s) = (\infty_-)-(\infty_+)
\end{equation*}

We compute the logarithmic derivatives of $s$.
\begin{equation*}
\frac{\dplus s}{s}=1 \text{ and } \frac{\dminus s}{s}=-1
\end{equation*}

Hence:
\begin{equation*}
i(\delta)=\left(\dotsc,b_i^2,\dotsc, 2, \dotsc, 0, \dotsc\right)
\end{equation*}

In particular, we recover that as long as some $e_i$ is greater than one, then $D(t)$ is not Pellian.

\section{Non squarefree case: the $W_l$}\label{sec:nonsquarefreeWl}

\subsection{First maps}

Let $h=g+\sum_{i=1}^g (e_i-1)$ (the degree of $D_1(t)$). For a point $P \in \Cv \dif S$, we denote with $[P]$ the image of the divisor $(P)-(\infty_+)$ in $G \cong \Gm^g \times \Ga^{h-g}$. We have that:

\begin{equation*}
[P]=(P)-(\infty_+)= \div \left( s-s(P) \right) \text{ as long as } P \neq \infty_+
\end{equation*}
\begin{equation*}
[\infty^+] \text{ is the trivial divisor}
\end{equation*}

If we denote by $z=s(P)$, the image of $[P]$ via $i$ is:

\begin{equation*}
i([P])=\left( \dotsc , \frac{z-b_i}{z-b_i^{-1}}, \dotsc , \dplus^{k} \left(\frac{\dplus (s-z)}{s-z}\right) (\xi^+_i)-(-1)^k\dminus^{k}\left(\frac{\dminus (s-z)}{s-z}\right)(\xi^-_i), \dotsc \right)
\end{equation*}

The first values of $\dplus^{k} \left(\frac{\dplus (s-z)}{s-z}\right)$ are:
\begin{equation*}
\left(\frac{\dplus (s-z)}{s-z}\right)=\frac{s}{s-z}=F_0(s,z)
\end{equation*}
\begin{equation*}
\dplus^{1} \left(\frac{\dplus (s-z)}{s-z}\right)=-\frac{sz}{(s-z)^2}=F_1(s,z)
\end{equation*}
\begin{equation*}
\dplus^{2} \left(\frac{\dplus (s-z)}{s-z}\right)=\frac{sz(s+z)}{(s-z)^3}=F_2(s,z)
\end{equation*}
\begin{equation*}
\dplus^{3} \left(\frac{\dplus (s-z)}{s-z}\right)=-\frac{sz(s^2+4sz+z^2)}{(s-z)^4}=F_3(s,z)
\end{equation*}

We observe that the denominator is $(s-z)^{k+1}$ and that the numerator is a homogenous polynomial in $s$ and $z$ of degree $k+1$. We perform the same computation for the $\dminus$ part. Notice that $s-z=-\frac{s^{-1}-z^{-1}}{s^{-1}z^{-1}}$.
\begin{equation*}
\left(\frac{\dminus (-\frac{s^{-1}-z^{-1}}{s^{-1}z^{-1}})}{-\frac{s^{-1}-z^{-1}}{s^{-1}z^{-1}}}\right)=\frac{z^{-1}}{s^{-1}-z^{-1}}=-\frac{s}{s-z}=G_0(s,z)
\end{equation*}
\begin{equation*}
\dminus^1 \left(\frac{\dminus (-\frac{s^{-1}-z^{-1}}{s^{-1}z^{-1}})}{-\frac{s^{-1}-z^{-1}}{s^{-1}z^{-1}}}\right)=-\frac{s^{-1}z^{-1}}{(s^{-1}-z^{-1})^2}=-\frac{sz}{(s-z)^2}=G_1(s,z)
\end{equation*}
\begin{equation*}
\dminus^2 \left(\frac{\dminus (-\frac{s^{-1}-z^{-1}}{s^{-1}z^{-1}})}{-\frac{s^{-1}-z^{-1}}{s^{-1}z^{-1}}}\right)=\frac{s^{-1}z^{-1}(s^{-1}+z^{-1})}{(s^{-1}-z^{-1})^3}=-\frac{sz(s+z)}{(s-z)^3}=G_2(s,z)
\end{equation*}
\begin{equation*}
\dminus^3 \left(\frac{\dminus (-\frac{s^{-1}-z^{-1}}{s^{-1}z^{-1}})}{-\frac{s^{-1}-z^{-1}}{s^{-1}z^{-1}}}\right)=-\frac{s^{-1}z^{-1}(s^{-2}+4s^{-1}z^{-1}+z^{-2})}{(s^{-1}-z^{-1})^4}=-\frac{sz(s^2+4sz+z^2)}{(s-z)^4}=G_3(s,z)
\end{equation*}

We see that $F=(-1)^{k+1}G$. Notice that both $z=0$ and $z=\infty$ are zeroes of $F$ and $G$, unless $k=0$, where $z=0$ gives $1$ and $-1$ respectively. If we evaluate them to $\xi_i^{\pm}$ accordingly, then the coordinate in $\Ga$ is (for instance for $k=3$):
\begin{equation*}
-\frac{b_iz(b_i^2+4b_iz+z^2)}{(b_i-z)^4}-\frac{b_i^{-1}z(b_i^{-2}+4b_i^{-1}z+z^2)}{(b_i^{-1}-z)^4}=F_3(b_i,z)-F_3(b_i^{-1},z)
\end{equation*}

The image of $[\infty^+]$ is $(\dotsc,1,\dotsc,0,\dotsc)$. We had already observed that the image of $\delta$ is:

\begin{equation*}
i(\delta)=(\dotsc, b_i^2, \dotsc , 1, \dotsc , 0 , \dotsc )
\end{equation*}

\subsection{The $W_l$}

We can now define the $W_l$ as before. We fix a nonnegative integer $l \le g$ and we have maps

\begin{equation*}
\begin{tikzcd}[column sep=3.2em]
(\Cv \dif S)^l \arrow[r, "\phi_l"] & G  \arrow[r, "i"] & \mathbb{G}_m^g \times \Ga^{\sum_{i=1}^g (e_i-1)}
\end{tikzcd}
\end{equation*}

so that $\phi_l(P_1,\dotsc,P_l)=[P_1]+\dotsb+[P_l]$ and therefore
$$(i \circ \phi_l) (P_1,\dotsc,P_l)=\left(\dotsc,\prod_{j=1}^l \frac{z_j-b_i}{z_j-b_i^{-1}},\dotsc, \sum_{j=1}^l F_k(b_i,z_j)+(-1)^kF_k(b_i^{-1},z_j), \dotsc \right)$$
where $z_j=s(P_j)$. We set again $S^*=s(S)=\{b_1,b_1^{-1}, \dotsc , b_g,b_g^{-1} \}$. As $i$ and $s$ are isomorphisms, the map $i \circ \phi_l$ is conjugated to the map:

\begin{equation*}
\begin{tikzcd}[column sep=3.2em]
\psi_l: (\mathbb{P}_1(\mathbb{C}) \setminus S^*)^l \arrow[r] & \mathbb{G}_m^g \times \Ga^{\sum_{i=1}^g (e_i-1)}
\end{tikzcd}
\end{equation*}

\begin{equation*}
\psi_l(z_1,\dotsc,z_l)=\left( \dotsc ,\prod_{j=1}^l \frac{z_j-b_i}{z_j-b_i^{-1}}, \dotsc, \sum_{j=1}^l F_k(b_i,z_j)+(-1)^kF_k(b_i^{-1},z_j), \dotsc \right)
\end{equation*}

We define $W_l$ to be the image of $\phi_l$. $W_l$ is again a constructible set (since the ground field is $\C$).

\section{Non squarefree case: Geometry of partial quotients}\label{sec:nonsquarefreeGeo}

As in the squarefree case, the $W_l$ are related to the degrees of the partial quotients and we can recover some analogs of the statements proven before.

\subsection{Preliminary lemmas}
The proofs are essentially the same as in the squarefree case. Lemmas \ref{lem:qnqn1}, \ref{lem:pre2} and \ref{lem:pre3} hold as above. We just review the proof of Lemma \ref{lem:pre1}. Let $f(t)=p(t)+uq(t)=P(s+\frac{1}{s})+(s-\frac{1}{s})Q(s+\frac{1}{s})$. We want to show that $\div f$ is zero in $G$ if and only if $D_1(t)$ divides $q(t)$.

\begin{proof}
Let us first compute the logarithmic derivative of $f$. We have that:
\begin{equation*}
\frac{\dplus f}{f}=\frac{\left(s-\frac{1}{s}\right)P'\left(s+\frac{1}{s}\right)+\left(s+\frac{1}{s}\right)Q\left(s+\frac{1}{s}\right)+\left(s-\frac{1}{s}\right)^2Q'\left(s+\frac{1}{s}\right)}{P\left(s+\frac{1}{s}\right)+\left(s-\frac{1}{s}\right)Q\left(s+\frac{1}{s}\right)}=
\end{equation*}
\begin{equation*}
\frac{uP(t)P'(t)-u^2P'(t)Q(t)+tP(t)Q(t)-utQ(t)^2+P(t)Q'(t)-uQ(t)Q'(t)}{P(t)^2-u^2Q(t)}
\end{equation*}

And we observe that, by induction, the successive derivatives have $P(t)$ as coefficient of $Q'(t),Q''(t) \dotsc$. Suppose $D_1(t)$ divides $q(t)$. Then by using the isomorphism $i$:

\begin{equation*}
\left(\dotsc,\frac{f(\xi_{\rho_i}^+)}{f(\xi_{\rho_i}^-)},\dotsc,\dplus^{k}\left(\frac{\dplus f}{f}\right)(\xi^+_i)-(-1)^k\dminus^{k}\left(\frac{\dminus f}{f}\right)(\xi^-_i), \dotsc\right)
\end{equation*}

We see that these vanish since $D_1(t)$ divides $Q(t)$ and the terms not containing derivatives of $Q(t)$ only occur with $u$ if $k$ is even and without if it is not. Vice-versa, if the image is zero, then we obtain that the $k+1$-th derivative of $Q(t)$ vanishes at $\rho_i$ for each $i$ (notice that $P(\rho_i) \neq 0$ since $\div f$ has support disjoint from $S$).
\end{proof}

\subsection{Translation over $G$}

We now give a weaker version of Theorem \ref{thm:thmpartialquotwl}, which relates the powers of $\delta$ lying in the $W_l$ and the convergents $(p(t),q(t))$ with $p(t)$ coprime with $D_1(t)$.

\begin{thm}\label{thm:thmpartialquotwlnonsqr}
Both $\delta^n \in W_l \dif W_{l-1}$ and $\delta^{n-1} \not \in W_{l-1}$ occur simultaneously if and only if there exist convergents $(p(t),q(t))$ of $\sqrt{D(t)}$ such that:\\
i) $\deg$ $p(t)=n$;\\
ii) the order of $p(t)-uD_1(t)q(t)$ at $(\infty_+)$ is $\deg$ $q(t)+d-l$;\\
iii) $p(t)$ and $D_1(t)$ are relatively prime.
\end{thm}

\begin{proof}

Step I: we first prove that if we have convergents satisfying the properties above, then $\delta^n \in W_l$.\\

Suppose $(p(t),q(t))$ are convergents with such properties. Then:
\begin{equation*}
\text{ div}(p(t)-uD_1(t)q(t))=(\text{deg }q(t) + d-l)(\infty_+)-k(\infty_-)+\sum_{i=1}^{k-\text{deg }q(t)-d+l} (x_i)
\end{equation*}

But $k=n$ by summing with its conjugate and $\deg p(t)=\deg q(t)+d$ so the latter equality becomes:
\begin{equation*}
\text{ div}(p(t)-uD_1(t)q(t))=(n-l)(\infty_+)-n(\infty_-)+\sum_{i=1}^{l} (x_i)
\end{equation*}

Since $p(t)$ and $D_1(t)$ are relatively prime, none of the $x_i$ belongs to $S$. Then, transposing this equality in $G$, the left hand side is zero by Lemma \ref{lem:pre1}, so we have:
\begin{equation*}
\delta^n = [x_1]+\dotsb+[x_l] \in W_l
\end{equation*}

Step II: we prove that if $\delta^n \in W_l$, then we can either find suitable convergents or we have that $\delta^n \in W_{l-1}$ or $\delta^{n-1} \in W_{l-1}$.\\

Let $\delta^n \in W_l$, so that $\delta^n=[x_1]+\dotsb+[x_l]$ with the $x_i \in \Cv \dif S$. We take the rational function with divisor
\begin{equation*}
(n-l)(\infty_+)-n(\infty_-)+\sum_{i=1}^{l} (x_i)
\end{equation*}

(notice that this can be done because we are in the genus zero case). The $x_i$ are neither $\infty_+$ nor $\infty_-$ provided that $\delta^n \not \in W_{l-1}$ and $\delta^{n-1} \not \in W_{l-1}$. Such function is regular on the affine part of $\mathcal{C}$ and its divisor maps to zero in $G$, so, using the Lemma \ref{lem:pre1}, it is of the form $p(t)-uD_1(t)q(t)$. As before, by considering the conjugate, we observe that $\deg p(t)=n$, $\deg q(t)=n-d$ and $p(t)-uD_1(t)q(t)$ vanishes at $\infty_+$ with order $\deg q(t)+d-l$. By Lemma \ref{lem:pre3}, $(p(t),q(t))$ are of the form $(P(t)r(t),Q(t)r(t))$ for convergents $(P(t),Q(t))$ and some polynomial $r(t)$ of degree $h'$. Notice that $P(t)$ is relatively prime with $D_1(t)$, since $P(t)r(t)$ is. But then, by Step I, the existence of the convergents $(P(t),Q(t))$ implies that $\delta^{n-h'} \in W_{l-h'}$. If $h' \ge 1$, then $\delta^{n-1} \in W_{l-1}$. If $h=0$, then $(p(t),q(t))$ are the required convergents.\\

Step III: we combine Step I and Step II to prove that, in the situation of Step I, actually $\delta^n \not \in W_{l-1}$ and $\delta^{n-1} \not \in W_{l-1}$.\\

Suppose that $\delta^n \in W_{l-1}$. Then we can find convergents $(P(t),Q(t))$ for $\sqrt{D(t)}$ such that $\deg P(t)=n, \deg Q(t)=n-d$ and $P(t)-Q(t) \sqrt{D(t)}$ vanishes with order $n-l+1$. But then $(p(t),q(t))$ would not be convergents: they should be equal to $(P(t),Q(t))$ because they have the same degree, but they give a different order of vanishing.\\

Suppose that $\delta^{n-1} \in W_{l-1}$. Then we can find convergents $(P(t),Q(t))$ for $\sqrt{D(t)}$ such that $\deg P(t)=n-1, \deg Q(t)=n-1-d$ and $P(t)-Q(t) \sqrt{D(t)}$ vanishes with order $n-1-l+1=n-l$. In particular, the partial quotient corresponding to $(P(t),Q(t))$ has degree $d+1-l \ge 2$, hence there can not be a partial quotient $(p(t),q(t))$ where $\deg p(t)=n$.
\end{proof}

\section*{Acknowledgement}
We thank Umberto Zannier for suggesting us this problem and for many precious comments.

We thank the Centro di Ricerca Matematica Ennio de Giorgi and the INdAM research group GNSAGA for financial support.

\end{document}